\documentclass[a4paper]{amsart}

\usepackage{amssymb}
\usepackage{amscd}
\usepackage{enumitem}
\usepackage{hyperref}
\usepackage[utf8]{inputenc}
\usepackage{newunicodechar}
\usepackage{varioref}
\usepackage[arrow,curve,matrix]{xy}
\usepackage{comment}
\usepackage{xifthen}

\usepackage{colortbl}
\usepackage{graphicx}
\usepackage{tikz}

\usepackage{mathrsfs}

\definecolor{linkred}{rgb}{0.7,0.2,0.2}
\definecolor{linkblue}{rgb}{0,0.2,0.6}

\numberwithin{figure}{section}

\usepackage[hyperpageref]{backref}

\newcommand{\Zar}{\mathrm{Zar}}
\newcommand{\cdh}{\mathrm{cdh}}
\newcommand{\cdp}{\mathrm{cdp}}
\newcommand{\sdh}{\mathrm{sdh}}
\newcommand{\h}{\mathrm{h}}
\newcommand{\eh}{\mathrm{eh}}
\newcommand{\rh}{\mathrm{rh}}
\newcommand{\et}{\mathrm{et}}

\newcommand{\tensor}{\otimes}
\newcommand{\Frac}{\mathsf{Frac}} 
\newcommand{\RZ}{\mathsf{RZ}} 
\newcommand{\Gal}{\mathsf{Gal}} 

\DeclareMathOperator{\rank}{rk}
\DeclareMathOperator{\trdeg}{tr{.}d}

\newcommand{\Sch}[1][k]{\mathsf{Sch}_{#1}}
\newcommand{\Schft}[1][k]{\mathsf{Sch}^{\mathsf{ft}}_{#1}}

\newcommand{\Shv}{\mathsf{Shv}}
\newcommand{\PreShv}{\mathsf{PreShv}}

\newcommand{\dvr}{\mathrm{dvr}}

\newcommand{\Spec}{\mathrm{Spec}}

\newcommand{\red}{\mathrm{red}}

\newcommand{\sn}{\mathrm{sn}}
\newcommand{\wn}{\mathrm{wn}}
\newcommand{\Ker}{\mathrm{ker}}
\newcommand{\an}{\mathrm{an}}
\newcommand{\reg}{\mathrm{reg}}
\renewcommand{\mod}{\textrm{-mod}}

\newcommand{\tor}{\mathrm{tor}}

\newcommand{\av}[1][\cdot]{|#1|}
\newcommand{\Th}{\mathcal{T}}

\newcommand{\rep}[2]{%
h_{#1\ifthenelse{\equal{#2}{}}{}{,#2}}%
}

\theoremstyle{theorem}
\newtheorem*{theoUn}{Theorem}
\newtheorem{thm}{Theorem}[section]

\newtheorem{coro}[thm]{Corollary}
\newtheorem{cor}[thm]{Corollary}
\newtheorem*{corUn}{Corollary}
\newtheorem{lemm}[thm]{Lemma}
\newtheorem{lemma}[thm]{Lemma}
\newtheorem{key}[thm]{Key Lemma}

\newtheorem{prop}[thm]{Proposition}

\newtheorem*{schoUn}{Scholium}

\theoremstyle{definition}
\newtheorem{defi}[thm]{Definition}
\newtheorem{defn}[thm]{Definition}

\newtheorem{rema}[thm]{Remark}
\newtheorem{rem}[thm]{Remark}

\newtheorem{ques}[thm]{Question}

\DeclareSymbolFontAlphabet{\scr}{rsfs}
\newcommand{\Fh}{\mathcal{F}}

\newcommand{\cK}{\mathcal{K}}
\newcommand{\Oh}{\mathcal{O}}
\newcommand{\OO}{\mathcal{O}}
\newcommand{\Hs}{\mathscr{H}}

\newcommand{\Z}{\mathbb{Z}}

\newcommand{\RR}{\mathbb{R}}
\renewcommand{\AA}{\mathbb{A}}
\newcommand{\Q}{\mathbb{Q}}

\newcommand{\m}{\mathfrak{m}}
\newcommand{\p}{\mathfrak{p}}
\newcommand{\q}{\mathfrak{q}}

\newcommand{\val}{\mathrm{val}}
\newcommand{\valone}{{\mathrm{val}^{\leq 1}}}
\newcommand{\valr}[1]{{\mathrm{val}^{\leq {#1}}}}
\newcommand{\valbig}{{\mathrm{val}^{\mathrm{big}}}}
\newcommand{\hval}{\mathrm{hval}}
\newcommand{\shval}{\mathrm{shval}}
\newcommand{\rval}{\mathrm{rval}}
\newcommand{\loc}{\mathrm{loc}}

\DeclareMathOperator{\eq}{eq}
\DeclareMathOperator{\coeq}{coeq}

\DeclareMathOperator{\image}{Image}
\DeclareMathOperator{\uSpec}{\underline{Spec}}

\title[Differential forms in positive characteristic II]{Differential forms in positive characteristic II: cdh-descent via functorial Riemann--Zariski spaces}
\author{Annette Huber}
\author{Shane Kelly}

\begin{document}

\vspace*{-5\baselineskip}

\maketitle

\begin{abstract}
This paper continues our study of the sheaf associated to Kähler differentials in the $\cdh$-topology and its cousins, in positive characteristic, without assuming resolution of singularities. The picture for the sheaves themselves is now fairly complete. 
We give a calculation $\OO_{\cdh}(X) \cong \OO(X^{\sn})$ in terms of the seminormalisation. We observe that the category of representable $\cdh$-sheaves is equivalent to the category of seminormal varieties. We conclude by proposing some possible connections to Berkovich spaces, and $F$-singularities in the last section.
The tools developed for the case of differential forms also apply in other contexts and should be of independent interest.
\end{abstract}



\section{Introduction}
This paper continues the program started in \cite{HJ} (characteristic $0$) and
\cite{HKK} (positive characteristic). For a survey see also \cite{porto}. 

\subsection*{Programme}

Let us quickly summarise the main idea. 
Sheaves of differential forms are very rich sources of invariants in the study of algebraic invariants of smooth algebraic varieties. However, they are much 
less well-behaved for singular varieties. In characteristic $0$, the use of the $\h$-topology---replacing Kähler differentials with their sheafification in this Grothendieck topology---is very successful. It unifies several ad-hoc notions 
and simplifies arguments. In positive characteristic, resolution of singularities would imply that the $\cdh$-sheafification could be used in a very similar 
way. 
Together with the results of \cite{HKK}, we now have a fairly complete unconditional picture, at least for the sheaves themselves. We refer to the follow-up
\cite{HK} for results on cohomological descent, where, however, many questions remain open.


\subsection*{Results} 

There are a number of weaker cousins of the $\h$-topology in the literature. They exclude Frobenius but still allow abstract blowups. The $\cdh$-topology (see Section~\ref{sec:top}) is the most well-established, appearing prominently in work on motives, $K$-theory, and having connections to rigid geometry, cf. \cite{FV, Voev00, SV00, CD15, CHSW, GH10, Cis13, KST, Mor16} for example. 
We write $\Omega^n_\tau$ for the sheafification  of the presheaf $X {\mapsto} \Gamma(X, \Omega^n_{X/k})$ with respect to the topology~$\tau$. 

\begin{theoUn}
Suppose $k$ is a perfect field. 
\begin{enumerate}

 \item \label{item:smooth}(\cite{HKK}, also Thm~\ref{thm:compare}) For a smooth $k$-scheme $X$, and $n \geq 0$
\begin{equation*} 
 \Gamma(X, \Omega^n_{X/k}) = \Omega^n_\cdh(X). 
 \end{equation*}

 \item (Thm~\ref{thm:coheren}, Thm~\ref{prop:etcoh}) For any finite type separated $k$-scheme $X$, the restriction $\Omega^n_\cdh|_{X_\Zar}$ (resp. $\Omega^n_\cdh|_{X_\et}$) to the small Zariski (resp. étale) site is a coherent $\OO_X$-module.

 \item (Prop.~\ref{prop:ehSN}) For functions, we have the explicit computation,
\begin{equation*} \label{equa:cdhsn}
\Oh_\cdh(X)=\Oh(X^\sn)
\end{equation*}
where $X^\sn$ is the seminormalisation of the variety $X$, see Section~\ref{sec:seminormal}. %
\item \label{item:top}(Prop.~\ref{prop:top}) For top degree differentials we have
\begin{equation*} 
 \Omega^d_\cdh(X ) \cong \varinjlim_{\substack{X' \to X \\\textrm{proper birational}}} \Gamma(X', \Omega^d_{X' / k}), \qquad \dim X = d. 
 \end{equation*}

\end{enumerate}
\end{theoUn}
By combining (\ref{item:smooth}) and (\ref{item:top}) we deduce a corollary that involves only Kähler differentials. To our knowledge the formula is new:
\begin{corUn}(Cor.~\ref{coro:mainDim}) Let $k$ be a perfect field and $X$ a smooth $k$-scheme. 
\begin{equation*} 
 \Gamma(X, \Omega^d_{X / k}) \cong \varinjlim_{\substack{X' \to X \\\textrm{proper birational}}} \Gamma(X', \Omega^d_{X' / k}), \qquad \dim X = d. 
 \end{equation*}
\end{corUn}

In contrast to the case of characteristic $0$, the sheaves $\Omega^n_{\cdh}$ are
not torsion-free (this was shown in \cite[Example 3.6]{HKK} by ``pinching'' along the Frobenius of a closed subscheme). So not only does lacking access to resolution of singularities cause proofs to become harder, the existence of inseparable field extensions actually changes some of the results.

\subsection*{Main tool} 

Our main tool, taking the role of a desingularisation of a variety $X$, is the category
\[ \val(X), \]
a functorial variant of the Riemann-Zariski space which we now discuss. Recall that the Rieman-Zariski space of an integral variety $X$ is (as a set) given by the set of (all, not necessarily discrete) $X$-valuation rings of $k(X)$. The Riemann-Zariski space is only functorial for dominant morphisms of integral varieties. 
We replace it by the category $\val(X)$ (see Definition~\ref{defn:val}) of
$X$-schemes of the form $\Spec(R)$ with $R$ either a field of finite transcendence degree over $k$ or a valuation ring of such a field. In \cite{HKK}, we were focussing on the discrete valuation rings---this turned out to be useful, but allowing non-noetherian valuation rings yields the much better tool. 

We define $\Omega^n_\val(X)$ as the set of global sections of the presheaf
$\Spec(R) {\mapsto} \Omega^n_{R/k}$ on $\val(X)$; that is
\[ \Omega^n_\val(X) = \varprojlim_{R \in \val(X)} \Omega^n_{R/k}. \]
 A global section admits the following very explicit description (Lemma~\ref{lemm:rsF}). It is uniquely determined by specifying an element $\omega_x\in \Omega^n_{\kappa(x)/k}$ for every point $x\in X$ subject to two compatibility conditions: If $R$ is an $X$-valuation ring of a residue field $\kappa(x)$, then
$\omega_x$ has to be integral, i.e., contained in $\Omega^n_{R/k} \subseteq \Omega^n_{\kappa(x)/k}$. If $\xi$ is the image of its special point in $X$, then
the $\omega_R|_{\Spec(R/\m)}$ has to agree with $\omega_\xi|_{\Spec(R/\m)}$ in $\Omega^n_{(R/\m)/k}$. 
The above mentioned results
are deduced by establishing $\eh$-descent (and therefore $\cdh$- and $\rh$-descent) for $\Omega^n_\val$. This is used to show:
\begin{theoUn}[{Theorem~\ref{thm:compare}}]
\[ \Omega^n_\rh = \Omega^n_\cdh = \Omega^n_\eh =\Omega^n_\val,\]
\end{theoUn}
\noindent giving a very useful characterisation of $\Omega^n_\cdh$, and answering the open question \cite[Prop.5.13]{HKK}.
We also show that using other classes of valuation rings (e.g., rank one or strictly henselian or removing the transcendence degree bound) produces the same sheaf, cf. Proposition~\ref{prop:omega_equal}.

\subsection*{Special cases and the $\sdh$-topology} %
There are two special cases, both of particular importance and with better 
properties: the case of $0$-forms, i.e., functions and the case of the ``canonical sheaf'', i.e., $d$-forms on the category of $k$-schemes of dimension at most $d$. In both cases, 
the resulting sheaves even have descent for the $\sdh$-topology introduced in \cite{HKK}, (Remark~\ref{rem:all_n=0} and Proposition~\ref{prop:sdh_top}):
\[ \Oh_\cdh=\Oh_\sdh,\hspace{4ex} \Omega^d_\cdh=\Omega^d_\sdh.\]
Recall that every variety is locally smooth in the $\sdh$-topology by de Jong's theorem on alterations, and the hope was that requiring morphisms to be separably decomposed would prohibit pathologies caused by purely inseparable extensions.  Unfortunately, $\Omega^n(X) \neq \Omega_{\sdh}^n(X)$ for general $n$, see \cite{HKK}.

\subsection*{Seminormalisation} %
For functions, we have the following explicit computation, see Proposition~\ref{prop:ehSN},
\begin{equation} \label{equa:cdhsn}
\Oh_\cdh(X)=\Oh(X^\sn)
\end{equation}
where $X^\sn$ is the seminormalisation of the variety $X$.
Here, the seminormalisation $X^\sn {\to} X$ is the universal %
morphism which induces isomorphisms of topological spaces and residue fields, see Section~\ref{sec:seminormal}. 
In fact, we have this result for all representable presheaves.

\begin{theoUn}[Proposition~\ref{prop:representable}] 
Suppose $S$ is a noetherian scheme and the normalisation of every finite type $S$-scheme is also finite type (i.e., that $S$ is Nagata cf.~Remark~\ref{rema:Nagata}), e.g., $S$ might be the spectrum of a perfect field. %
Then for all separated finite type $S$-schemes $X, Y$, the canonical morphisms
\begin{equation} \label{equa:asgddbgaszxcv}
h^Y_\rh(X)=
h^Y(X^{\sn})
\end{equation}
are isomorphisms, where $h^Y(-) = \hom_S(-, Y)$. 
The natural maps
\begin{equation} \label{equa:jksadfewrhqwe}
 h^Y_\rh\to h^Y_\cdh\to h^Y_\eh\to h^Y_\sdh\to h^Y_\val 
 \end{equation}
are isomorphisms of presheaves on $\Schft[S]$.
\end{theoUn}

In characteristic zero, Equation~(\ref{equa:asgddbgaszxcv}) is already formulated for the $\h$-topology by Voevodsky, see \cite[Section~3.2]{Voe96}, and generalised to algebraic spaces by Rydh, \cite{Ryd10}.  The theorem confirms that we have identified the correct analogy in positive characteristic.

As Equation~(\ref{equa:asgddbgaszxcv}) confirms, the $\cdh$-topology is not 
subcanonical. In fact, the full subcategory spanned by those $\cdh$-sheaves which are sheafifications of representable sheaves is equivalent to the 
category of seminormal schemes (Corollary~\ref{coro:hcdhleftAdjoint}). In particular, the subcategory of smooth or even normal varieties remains unchanged, however, we 
lose information about certain singularities, e.g., cuspidal singularities are
 smoothened out. Depending on the question, this might be considered an advantage or a disadvantage. We strongly argue that it is an advantage for the natural
 questions of birational algebraic geometry, where differential forms are a
 main tool.

Recall that if $X$ is integral, global sections of the normalisation $\widetilde{X}$ is the intersection of all $X$-valuation rings of the function field,
\[ \Gamma(\widetilde{X}, \OO_{\widetilde{X}}) = \bigcap_{\textrm{val.rings } R \textrm{ of } k(X)} R. \]
In light of $\OO_{\cdh} \cong \OO_{\val}$, Equation~(\ref{equa:cdhsn}) has the following neat interpretation:

\begin{schoUn}
If $X$ is reduced, global sections of the seminormalisation $X^{sn}$ is the ``intersection'' of all $X$-valuation rings,
\[ \Gamma(X^{sn}, \OO_{X^{sn}}) = \varprojlim_{\Spec(R) \to X, \textrm{ with } R \textrm{ a val.ring}} R. \]
\end{schoUn}

%


The seminormalisation was introduced and studied quite some time ago, see for example \cite{traverso} and \cite{swan}; for a historical survey see 
\cite{Vit11}. The original motivation for considering the seminormalisation (or rather, the closely related and equivalent in characteristic $0$ concept of
 weak normalisation) was to make the moduli space of positive analytic $d$-cycles on a projective variety ``more normal'' without changing its topology, 
i.e., without damaging too much the way that it solved its moduli
 problem, \cite{AN67}. Clearly, the $\cdh$- and $\eh$-topology are relevant to these moduli questions. Indeed, the $\cdh$-topology already appears in the 
study of moduli of cycles in \cite[Thm.4.2.11]{SV00}. In relation to this, 
let us point out that basically all of the present paper works for arbitrary 
unramified étale sheaves commuting with filtered limits of schemes, and in 
particular applying $(-)_\val$ to various Hilbert presheaves provides alternative constructions for Suslin and Voevodsky's relative cycle presheaves $z(X/S, r), z^{eff}(X/S, r), c(X/S, r), c^{eff}(X/S, r)$ introduced in \cite{SV00}, and heavily used in their work on motivic cohomology, but such 
matters go beyond the scope of this current paper.

\subsection*{Outline of the paper} %
We start in Section~\ref{sec:basics} by collecting basic notation and facts on the Grothendieck topologies that we use, seminormality and valuation rings. In Section~\ref{ssec:loc} we introduce our main tool: different categories of local rings above a given $X$ and discuss the relation to the Riemann-Zariski space.

In Section~\ref{sec:presheaves} we discuss and compare the presheaves on schemes of finite type over the base induced from presheaves on our categories of local rings. 
Everything is then applied to the case of differential forms.

In Section~\ref{sec:descent} we  verify sheaf conditions on these presheaves. In Theorem~\ref{thm:compare}, this culminates in our main comparison theorem on differential forms.

Section~\ref{sec:quasi-coherence} establishes coherence of $\Omega^n_\rh$ locally on the Zariski or even small \'etale site of any $X$.
In Section~\ref{sec:specialCases} we turn to the special examples of $\Oh$, the canonical sheaf, and more generally the category of representable sheaves.
Finally, in Section~\ref{sec:future} we outline interesting open connections to the theory of Berkovich spaces and to $F$-singularities.

\noindent {\bf Acknowledgements:} It is a pleasure to thank Manuel Blickle, Fumiharu Kato, Stefan Kebekus, Simon Pepin-Lehalleur, Kay R{\"u}lling, David Rydh, and Vasudevan Srinivas for inspiring discussions and answers to our mathematical questions.

\section{Commutative algebra and general definitions}\label{sec:basics}
\subsection{Notation}\label{ssec:not}
Throughout $k$ is assumed to be a perfect field. The case of interest is
the case of positive characteristic. Sometimes we will use a separated noetherian base scheme $S$. This includes the case $S = \Spec(k)$ of course.

The valuation rings we use are \emph{not} assumed to be noetherian!

We denote by 
 $\Schft[S]$ the category of separated schemes of finite type over $S$, and write $\Schft$ when considering the case $S = \Spec(k)$. 

We write $\Omega^n(X)$ for the vector space of $k$-linear $n$-differential forms on a $k$-scheme
$X$, often denoted elsewhere by $\Gamma(X, \Omega^n_{X/S})$. Note, the assignment $X\mapsto \Omega^n(X)$ is functorial in $X$.

Following \cite[Tag 01RN]{stacks-project}, we call a morphism of schemes with finitely many irreducible components $f:X\to Y$ \emph{birational} if it induces a bijection between the sets of irreducible components and an isomorphism on the residue fields of generic points. In the case of varieties this is equivalent to the existence of a dense open subsets $U\subseteq X$ and $V\subset Y$ such that $f$ induces an isomorphism $U\to V$.

\subsection{Topologies}\label{sec:top}\label{sec:top}

\begin{defi}[$\cdp$-morphism]\label{defi:cdp}
  A morphism $f : Y\to X$ is called a \emph{$\cdp$-morphism} if it is proper and
  completely decomposed, where by ``completely decomposed'' we mean that for
  every---not necessarily closed---point $x \in X$ there is a point $y \in Y$ with $f(y) = x$ and $[k(y): k(x)]
  = 1$.
\end{defi}

These morphisms are also referred to as \emph{proper $\cdh$-covers} (e.g., by Suslin-Voevodsky in \cite{SV00}), or
\emph{envelopes} (e.g., by Fulton in \cite{Ful}). 

\begin{rema}[$\rh$, $\cdh$ and $\eh$-topologies]\label{rem:defi:rhetc}\ 
\begin{enumerate}
 \item Recall that the $\rh$-topology on $\Schft[S]$ is generated by the Zariski
  topology and $\cdp$-morphisms, \cite{GL}.  In a similar vein, the
  $\cdh$-topology is generated by the Nisnevich topology and $\cdp$-morphisms,
  \cite[§~5]{SV00}.  The $\eh$-topology is generated by the étale topology and
  $\cdp$-morphisms, \cite{Gei06}.
  
  \item \label{rem:defi:rhetc:refine} We even have a stronger statement: Every $\rh$- (resp. $\cdh$-, $\eh$-) covering $U {\to} X$ in $\Schft[S]$ admits a refinement of the form $W {\to} V {\to} X$ where $V {\to} X$ is a $\cdp$-morphism, and $W \to V$ is a Zariski (resp. Nisnevich, étale) covering, \cite[(Proof of)Prop.5.9]{SV00}.
\end{enumerate}
\end{rema}

\subsection{Seminormality}\label{sec:seminormal}

A special case of a $\cdp$-morphism is the seminormalisation.

\begin{defn}[\cite{swan}]\label{defn:sn} 
A reduced ring $A$ is \emph{seminormal} if for all $b, c \in A$ satisfying $b^3 = c^2$, there is an $a \in A$ with $a^2 = b, a^3 = c$. Equivalently, every morphism $\Spec(A) \to \Spec(\Z[t^2, t^3])$ factors through $\Spec(\Z[t]) \to \Spec(\Z[t^2, t^3])$.
\end{defn}



\begin{defi}
Recall that an inclusion of rings $A \subseteq B$ is called \emph{subintegral} if $\Spec(B) \to \Spec(A)$ is a completely decomposed homeomorphism, \cite[§2]{swan}. In other words, if it induces an isomorphism on topological spaces, and residue fields.
\end{defi}

\begin{rema}
Any (not necessarily injective) ring map $\phi: A \to B$ inducing a completely decomposed homeomorphism is \emph{integral} in the sense that for every $b \in B$ there is a monic $f(x) \in A[x]$ such that $b$ is a solution to the image of $f(x)$ in $B[x]$ \cite[Tag 04DF]{stacks-project}. Consequently, a postiori, subintegral inclusions are contained in the normalisation.
\end{rema}

We have the following nice properties. 

\begin{lemm} \label{lemm:seminormalProperties}
Let $A$, resp. $(A_i)_{I}$, be a reduced ring, resp. (not necessarily filtered) diagram of reduced rings.
\begin{enumerate}
 \item If $A$ is normal, it is seminormal.
 \item \label{lemm:seminormalProperties:limits} \cite[Cor.3.3]{swan} If all $A_i$ are seminormal, then so is $\varprojlim_{i \in I} A_i$.
 \item \label{lemm:seminormalProperties:subInt} \cite[Cor.3.4]{swan} If the total ring of fractions $Q(A)$ is a product of fields, then $A$ is seminormal if and only if for every subintegral extension $A \subseteq B \subseteq Q(A)$ we have $A = B$. In particular, this holds if $A$ is noetherian or $A$ is valuation ring.
 \item \label{lemm:seminormalProperties:universal} \cite[§2, Thm~4.1]{swan} There exists a universal morphism $A \to A^{\sn}$ with target a seminormal reduced ring. The morphism $\Spec(A^{\sn}) \to \Spec(A)$ is subintegral. 
 \item \label{lemm:seminormalProperties:localise} \cite[Cor.4.6]{swan} For any multiplicative set $S \subseteq A$ we have $A^{\sn}[S^{-1}] {=} A[S^{-1}]^{\sn}$, in particular, if $A$ is seminormal, so is $A[S^{-1}]$.
\end{enumerate}
\end{lemm}

\begin{rema} \label{rema:Nagata}
Recall that a scheme $S$ is called \emph{Nagata} if it is locally noetherian, and if for every $X \in \Schft[S]$, the normalisation $X^n {\to} X$, cf. \cite[Tag 035E]{stacks-project}, is finite. Well-known examples are fields or the ring of integers $\Z$, or more generally, quasi-excellent rings are Nagata, \cite[Tag 07QV]{stacks-project}.

Actually, the definition of Nagata \cite[Tag 033S]{stacks-project} is: for every point $x \in X$, there is an open $U \ni x$ such that $R = \OO_X(U)$ is a Nagata ring \cite[Tag 032R]{stacks-project}, i.e., noetherian and for every prime $\p$ of $R$ the quotient $R / \p$ is N-2 \cite[Tag 032F]{stacks-project}, i.e., for every finite field extension $L / \Frac(R/\p)$, the integral closure of $R / \p$ in $L$ is finite over $R / \p$. However, Nagata proved \cite[Tag 0334]{stacks-project} that Nagata rings are characterised by being noetherian and universally Japanese \cite[Tag 032R]{stacks-project}, but being universally Japanese is the same as: for every finite type ring morphism $R \to R'$ with $R'$ a domain, the integral closure of $R'$ in its fraction field is finite over $R'$, \cite[Tag 032F, Tag 0351]{stacks-project}. Since we can assume the $U$ from above is affine, and finiteness of the normalisation is detected locally, one sees that our definition above is equivalent to the standard one.
\end{rema}

\begin{prop} \label{prop:sn}
Let $X$ be a scheme. There exists a universal morphism 
\[ X^\sn\to X\]
from a scheme whose structure sheaf is a sheaf of seminormal reduced rings, called the \emph{seminormalisation} of $X$. 
\begin{enumerate}
 \item It is a homeomorphism of topological spaces.
 \item If $X = \Spec(A)$ then $X^{\sn} = \Spec(A^{\sn})$.
 \item If the normalisation $X^n {\to} X$ is finite (e.g., if $X \in \Schft$ or more generally, if $X \in \Schft[S]$ with $S$ a Nagata scheme), then $X^\sn {\to} X$ is a finite $\cdp$-morphism.
\end{enumerate}
\end{prop}

\begin{proof}
Replacing $X$ with its associated reduced scheme, $X_{\red}$ we can assume that $X$ is reduced. Define $X^{\sn}$ to be the ringed space with the same 
underlying topological space as $X$, and structure sheaf the sheaf obtained from $U \mapsto \OO(U)^{\sn}$. It satisfies the appropriate universality by the 
universal property of $(-)^\sn$ for rings.

Since $(-)^{\sn}$ commutes with inverse limits and localisation, 
Lemma~\ref{lemm:seminormalProperties}\eqref{lemm:seminormalProperties:limits},\eqref{lemm:seminormalProperties:localise}, for any reduced ring $A$ the structure sheaf of $\Spec(A^{\sn})$ is a sheaf of seminormal rings, 
and we obtain a canonical morphism $\Spec(A^{\sn}) \to \Spec(A)^{\sn}$ of ringed spaces. By the universal properties, $(-)^{\sn}$ commutes with colimits, 
so for any point $x \in X$ we have $\OO_{X^{\sn}, x} = \OO_{X, x}^{\sn}$, and 
consequently, $\Spec(A^{\sn}) \to \Spec(A)^{\sn}$ is an isomorphism of
 ringed spaces. From this we deduce that in general $X^{\sn}$ is a 
scheme, and $X^{\sn} \to X$ is a completely decomposed morphism of schemes, cf. Lemma~\ref{lemm:seminormalProperties}\eqref{lemm:seminormalProperties:universal}.

Finally, by the universal property, there is a factorisation $X^n \to X^{\sn} \to X$. So if the normalisation if finite, then so is the seminormalisation.
\end{proof}

The following well-known property explains the significance of
the seminormalisation in our context.

\begin{lemma}\label{lem:sn-descent}
Let $\Fh$ be an $\rh$-sheaf on $\Schft[S]$, let $X \in \Schft[S]$ and $X' \to X$ a completely decomposed homeomorphism. Then
\[ \Fh(X)\to \Fh(X^\red)\to \Fh(X')\]
are isomorphisms. In particular, 
\[ \Fh(X) \cong \Fh(X^{\sn}) \]
if $X^{\sn} \in \Schft[S]$ (for example if $S$ is Nagata).
\end{lemma}
\begin{proof}
The map $X^\red\to X$ is $\cdp$ and we have
$(X^\red \times_X X^\red)=X^\red$. The isomorphism for $X^\red$ follows
from the sheaf sequence. 
The same argument applies to $X'$: Since $\Fh$ is an $\rh$-sheaf we can assume $X'$ is reduced. Then since $X' \to X$ is a finite, cf.~\cite[Tag 04DF]{stacks-project}, completely decomposed homeomorphisms, so are the projections $X' {\times_X} X' \to X'$, and it follows that the diagonal induces an isomorphism $X' \stackrel{\sim}{\to} (X' {\times_X} X')^{\red}$.
\end{proof}

\begin{lemm} \label{lemm:sncdpsubcan}
Suppose that 
\[ \xymatrix{
E \ar[r]^j \ar[d]_q & X' \ar[d]^p \\
Z \ar[r]_i & X 
} \]
is a commutative square in $\Schft[S]$, with $i, j$ a closed immersions, $p, q$ finite surjective, and $p$ an isomorphism outside $i(Z)$.

Then the pushout $Z \sqcup_E X'$ exists in $\Schft[S]$, is reduced if both $Z$ and $X'$ are reduced, and the canonical morphism $Z \sqcup_E X' \to X$ is a finite completely decomposed homeomorphism. In particular, 
\[ Z \sqcup_E X' \cong X \]
if $Z, X'$ are reduced and $X$ is seminormal.
\end{lemm}

\begin{proof}
First consider the case where $X$ is affine (so all four schemes are affine). The pushout exists by \cite[Sco.4.3, Thm.5.1]{Fer03} and is given explicitly by the spectrum of the pullback of the underlying rings. By construction it is reduced as $Z$ and $X'$ are. 
The underlying set of the pushout is the pushout of the underlying sets cf. \cite[Sco.4.3]{Fer03}.
As $p$ is an isomorphism outside $i(Z)$ and $q$ is surjective, it follows that $Z \sqcup_E X' \to X$ is a bijection on the underlying sets. %
The existence of the (continuous) map (of topological spaces) $Z \sqcup_E X' \to X$ shows that every open set of $X$ is an open set of $Z \sqcup_E X'$. 
Conversely, if $W$ is a closed set of $Z \sqcup_E X'$ then its 
preimage in $X'$ is closed. As $p$ is proper, this implies that $W$ is also closed in $X$. In other 
words, every open set of $Z \sqcup_E X'$ is an open set of $X$.
 Hence $Z \sqcup_E X' {\to} X$ is a homeomorphism with reduced source. %
 Now 
$Z \sqcup X' \to X$ is finite and completely decomposed, so $Z \sqcup_E X' \to X$ is also finite and completely decomposed. To summarise, $Z \sqcup_E X' \to X$ is a finite completely decomposed homeomorphism. That is, if both are reduced it comes from a subintegral extension of rings. If in addition $X$ is seminormal, this implies that it is an isomorphism, Lemma~\ref{lemm:seminormalProperties}\eqref{lemm:seminormalProperties:subInt}.

For a general $X \in \Schft[S]$, the pushout exists by \cite[Thm.7.1]{Fer03} (one checks the condition (ii) easily by pulling back an open affine of $X$ to $X'$). Then the properties of $Z \sqcup_E X' \to X$ claimed in the statement can be verified on an open affine cover of $X$. As long as $U {\times_X} (Z \sqcup_E X') \cong (U {\times_X}Z) \sqcup_{U {\times_X}E} (U {\times_X}X')$ for every open affine $U \subseteq X$, it follows from the case $X$ is affine. This latter isomorphism can be checked in the category of locally ringed spaces using the explicit description of \cite[Sco.4.3]{Fer03}; it is also a special case of \cite[Lem.4.4]{Fer03}.
\end{proof}


\subsection{Valuation rings}

Recall that an integral domain $R$ is called a \emph{valuation ring} if for all $x \in K = \Frac(R)$, at least one of $x$ or $x^{-1}$ is in $R \subseteq K$. If $R$ contains a field $k$, we will say that $R$ is a \emph{$k$-valuation ring}. We will say $R$ is a \emph{valuation ring of $K$} to emphasise that $K = \Frac(R)$.

The name valuation ring comes from the fact that the abelian group $\Gamma_R = K^* / R^*$ equipped with the relation ``$a \leq b$ if and only if $b / a \in (R{-}\{0\}) / R^*$'' is a totally ordered group, and the canonical homomorphism $v: K^* \to K^* / R^*$ is a valuation, in the sense that $v(a + b) \geq \min(v(a), v(b))$ for all $a, b \in K^*$. Conversely, for any valuation on a field, the set of elements with non-negative value are a valuation ring in the above sense. If $\Gamma_R = K^* / R^*$ is isomorphic to $\Z$ we say that $R$ is a \emph{discrete valuation ring}. Every noetherian valuation ring is either a discrete valuation ring or a field.

One of the many, varied characterisations of valuation rings is the following.

\begin{prop}[{\cite[Ch.VI §1, n.2, Thm.1]{Bou64}}] \label{prop:val_ordered}Let $R \subseteq K$ be a subring of a field. Then $R$ is a valuation ring if and only if its set of ideals is totally ordered.
\end{prop}

In particular, this implies that the maximal ideal is unique, that is, every valuation ring is a local ring. The cardinality of the set of non-zero prime ideals of a valuation ring is called its \emph{rank}. As the set of primes is totally ordered, the rank agrees with the Krull dimension. 

\begin{coro}\label{coro:is_valuation}
Let $R$ be a valuation ring. If $\p \subseteq R$ is a prime ideal, then both the quotient $R / \p$ and the localisation $R_\p$ are again, valuation rings.
\end{coro}

\begin{coro}\label{coro:locIsVal}
Let $R$ be a valuation ring and $S \subseteq R {-} \{0\}$ a multiplicative set. Then $R[S^{-1}]$ is a valuation ring. In fact, $\p = \bigcup_{S \cap \q = \varnothing} \q$ is a prime and $R[S^{-1}]{ =} R_\p$.
\end{coro}

\begin{proof}
By Corollary~\ref{coro:is_valuation} it suffices to show the second claim. Clearly $\p$ is prime. Recall, that the canonical $R{\to}R[S^{-1}]$ induces an isomorphism of \emph{locally ringed spaces} between $\Spec(R[S^{-1}])$, and $\{\q : \q \cap S {=} \varnothing\} \subseteq \Spec(R)$. Since this set has a maximal element, namely $\p$, the morphism $\Spec(R_\p) {\to} \Spec(R[S^{-1}])$ is an isomorphism of \emph{locally ringed spaces}. Applying $\Gamma(-, \OO_-)$ gives the the desired ring isomorphism.
\end{proof}

Another one of the many characterisation of valuation rings is the following.

\begin{prop}[{\cite[Tag 092S]{stacks-project}, \cite{Oli83}}] \label{prop:sub_flat}
A local ring $R$ is a valuation ring if and only if every submodule of every flat $R$-module is again a flat $R$-module.
\end{prop}

From this we immediately deduce the following, crucial for Corollary~\ref{coro:etaleValCover}.

\begin{coro} \label{coro:flatUnramifiedValuation}
Let $R$ be a valuation ring. %
If $R {\to} A$ is a flat $R$-algebra with $A {\otimes}_RA {\to} A$ also flat. Then for every prime ideal $\p \subset A$, the localisation $A_\p$ is again a valuation ring.
\end{coro}

\begin{proof}
If $A$ satisfies the hypotheses, then so does $A_\p$, so we can assume $A$ is local, and it suffices to show that every sub-$A$-module of a flat $A$-module is a flat $A$-module. Recall that flatness of $R\to A$ implies the forgetful functor $U:A\mod \to R\mod$ preserves flatness, because $- \otimes_A A \otimes_R - \cong (U-) \otimes_R-$. Recall also that flatness of $A{\otimes}_RA {\to}A$ implies that $U$ \emph{detects} flatness, because we have isomorphisms $- \otimes_A - \cong (- \otimes_A (A \otimes_R A) \otimes_A -) \otimes_{(A \otimes_R A)} A$ and $(- \otimes_A (A \otimes_R A) \otimes_A -) {\cong} (U-) \otimes_R (U-)$. Since $U$ preserves and detects flatness, and preserves monomorphisms, the claim now follows from Proposition~\ref{prop:sub_flat}.
\end{proof}

As one might expect, the rank is bounded by the transcendence degree.

\begin{prop}[{\cite[Ch.VI §10, n.3, Cor.1]{Bou64}, \cite[6.1.24]{GR03}, \cite[Cor.3.4.2]{EP05}}] \label{valuations}
Suppose that $R'$ is a valuation ring, $K' = \Frac(R')$, let $K' / K$ be a field extension. Note that $R = K \cap R'$ is again a valuation ring, and the inclusion $R \subseteq R'$ induces a field extension of residue fields $\kappa \subseteq \kappa'$, and a morphism $\Gamma_R \to \Gamma_{R'}$ of totally ordered groups. With this notation, we have 
\[ \trdeg (K'/K) + \rank R \geq \trdeg (\kappa'/\kappa) + \rank R'. \]
In particular, if $R'$ is a $k$-valuation ring for some field $k$, we have
\[ \trdeg (K'/k) \geq \trdeg (\kappa'/k) {+} \rank R', \]
and finiteness of $\trdeg (K'/k)$ implies finiteness of $\rank R'$.
\end{prop}

Recall that a local ring $R$ is called \emph{strictly henselian} if every faithfully flat étale morphism $R \to A$ admits a retraction.
Every local ring $R$ admits a ``smallest'' local morphism towards a strictly henselian local ring, which is unique up to non-unique isomorphism. The target of any such morphism is called \emph{the strict henselisation} and is denoted by $R^{sh}$. There are various ways to construct this. One standard construction is to choose a separable closure $\kappa^s$ of the residue field $\kappa$ and take the colimit
\begin{equation} \label{equa:strictHenselisation} 
R^{sh} = \varinjlim\limits_{R \to A \to \kappa^s} A
\end{equation}
 over factorisations such that $R \to A$ is étale. 

From Corollary~\ref{coro:flatUnramifiedValuation} we deduce the following.

\begin{coro} \label{coro:etaleValCover}
For any valuation ring $R$, the strict henselisation is again a valuation ring. If $R$ is of finite rank then: 
\begin{enumerate}
 \item \label{locaIsEtale} If $R \to A$ is an étale algebra, $R \to A_\p$ is also an étale algebra for all $\p \subset A$.

 \item \label{coro:etaleValCover:shval} In the colimit \eqref{equa:strictHenselisation}, it suffices to consider those étale $R$-algebras $A$ which are valuation rings.
 
 \item Every étale covering $U \to \Spec(R)$ admits a Zariski covering $V \to U$ such that $V$ is a disjoint union of spectra of valuation rings.

\end{enumerate}
\end{coro}

This corollary actually holds whenever the primes of $R$ are well-ordered by \cite[Ch.VI §8, n.3, Thm.1]{Bou64}, and might very well be true in general, but finite rank suffices for our purposes.

\begin{proof}
Recall that the diagonal is open immersion in the case of an unramified morphism. Hence the assumptions of Corollary~\ref{coro:flatUnramifiedValuation} are
satisfied for all \'etale $R$-algebras and their cofiltered limits.
In particular,  $R^{sh}$ is a valuation ring.
\begin{enumerate}
  \item It suffices to show that $A \to A_\p$ is of finite type. First note that since $\Spec(A) \to \Spec(R)$ is étale, it is quasi-finite, and so since $R$ has finite rank, $A$ has finitely many primes. For every prime $\q \subset A$ such that $\q \not\subseteq \p$, choose one $\pi_\q \in \q \backslash (\q \cap \p)$. Then $A_\p = A[\pi_{\q_1}^{-1}, \dots, \pi_{\q_n}^{-1}]$, cf. Corollary~\ref{coro:locIsVal}.

 \item We can replace each $A$ with $A_\p$ without affecting the colimit, cf. Part~\eqref{locaIsEtale}.
 
\item Part~\eqref{locaIsEtale} also implies that $\Spec(A_\p)\to \Spec(A)$ is an open immersion. Each of the finitely many $A_\p$ is valuation ring by Corollary~\ref{coro:flatUnramifiedValuation} \qedhere
\end{enumerate}
\end{proof}

Just as strictly henselian rings are ``local rings'' for the étale topology, strictly henselian valuation rings are the ``local rings'' for the $\eh$-topology.

\begin{prop}[{\cite[Thm.2.6]{GabKel}}] \label{prop:eh_val}
A $k$-ring $R$ is a valuation ring (resp. henselian valuation ring, resp. strictly henselian valuation ring) if and only if for every $\rh$-covering (resp. $\cdh$-covering, resp. $\eh$-covering) $\{U_i \to X\}$ in $\Schft$, the morphism of sets of $k$-scheme morphisms
\[ \coprod_{i \in I} \hom(\Spec(R), U_i) \to \hom(\Spec(R), X) \]
is surjective.
\end{prop}

The key input into the present paper is the same as for \cite{HKK}.

\begin{thm}\label{thm:hypV}For every  finitely generated extension
$K/k$ and every $k$-valuation ring $R$ of $K$ the map  
\[ \Omega^n(R)\to \Omega^n(K)\]
 is injective for all
$n\geq 0$, i.e., $\Omega^n$ is torsion-free on $\val(k)$.
\end{thm}
This is due to Gabber and Ramero for $n=1$, see  \cite[Corollary~6.5.21]{GR03}.
The general case is deduced in  \cite[Lemma~A.4]{HKK}.

\subsection{Presheaves on categories of local rings}\label{ssec:loc}
We fix a base scheme $S$. The case of main interest for the present paper
is $S=\Spec(k)$ with $k$ a field.

\begin{defn}\label{defn:val}
Let $X/S$ be a scheme of finite type. We will use the following notation.
\begin{enumerate}
\item[$\valbig(X)$] is the category
of $X$-schemes $\Spec(R)$ with $R$ either a 
field extension of $k$ or a $k$-valuation ring of such a field.

\item[$\val(X)$] is the full subcategory of $\valbig(X)$ of those $\pi: \Spec(R){\to} X$ such that the transcendence degree $\trdeg (k(\eta) / k(\pi(\eta))) $ is finite (but we do not demand the field extension to be finitely generated), where $\eta$ is the generic point of $\Spec(R)$.

\item[${\valr{r}(X)}$] is the full subcategory of $\val(X)$ of those $\Spec(R){\to} X$ such that rank of the valuation ring $R$ is $\leq r$.

\item[$\dvr(X)$] is the full subcategory of $\val(X)$ of those $\Spec(R){\to} X$ such that $R = \OO_{Y, y}$ for some $Y \in \Schft[X]$ and some point $y \in Y$ of codimension $\leq 1$ which is regular.

\item[$\shval(X)$] is the full subcategory of $\val(X)$ of those $\Spec(R){\to} X$ such that $R$ is strictly henselian.

\item[$\rval(X)$] is the full subcategory of $\val(X)$ of those $\Spec(R) {\to} X$ such that $R$ is a valuation ring of a residue field of $X$.
\end{enumerate}

We generically denote $\loc(X)$ one of the above categories of local rings.
We also write $\loc$ for $\loc(S)$.
\end{defn}

Note that the morphisms in the above categories are not required to be induced by local homomorphisms of local rings. All $X$-morphisms are allowed.

\begin{defn}\label{defn:torsion-free}Let $\Fh$ be a presheaf on $\loc$. 
We say that $\Fh$ is \emph{torsion-free}, if 
\[ \Fh(R)\to \Fh(\Frac(R)) \] is injective
for all valuation rings $R$ in $\loc$ and their field of fractions
$\Frac(R)$.
\end{defn}

We could have also called this property \emph{separated}, since when $\Fh$ is representable, it is the valuative criterion for separatedness.

\begin{lemma}\label{lem:val_sn}
For any $X \in \Schft[S]$, let $X'{\to}X$ be any completely decomposed homeomorphism (in particular, if $X^{\sn} \in \Schft[S]$, for example if $S$ is Nagata, we can take $X' = X^{sn}$), then
\[ \loc(X)=\loc(X')\]
for all of the above categories of local rings.
\end{lemma}
\begin{proof}
Follows from the universal property of $(-)^{\sn}$ since valuation rings are normal.
\end{proof}

The category $\val$ should be seen as a fully functorial version of
the Riemann-Zariski space.

\begin{defn}\label{defn:RZ}
Let $X$ be an integral $S$-scheme of finite type with generic point $\eta$. As a set, the
Riemann-Zariski space $\RZ(X)$, called the ``Riemann surface'' in \cite[§~17,
p.~110]{SZ}, is the set of (not necessarily discrete) valuation rings over $X$
of the function field $k(X)$ (see also \cite[Before Rem.2.1.1, After Rem.2.1.2, Before Prop.2.2.1, Prop.2.2.1, Cor.3.4.7]{Tem11}).  We turn it into a topological space by using
as a basis the sets of the form  
\[ E(A')=\{R\in\RZ(X)| A'\subset R\}\]
where $\Spec(A)$ is an affine open of $X$, and $A'$ is a finitely generated sub-$A$-algebra of $k(X)$ (cf. \cite[Before Lem.3.1.1, Before Lem.3.1.8]{Tem11}). It has a canonical structure of locally ringed space induced by the assignment $U \mapsto \bigcap_{R \in U} R$ for open subsets $U \subseteq \RZ(X)$. One can equivalently define $\RZ(X)$ as the inverse limit of all proper birational morphisms $Y \to X$, taking the inverse limit in the category of locally ringed spaces. In particular, as a set it is the inverse limit of the underlying sets of the $Y$, equipped with the coarsest topology making the projections $\varprojlim Y \to Y$ continuous, and the structure sheaf is the colimit of the inverse images of the $\OO_Y$ along the projections $\varprojlim Y \to Y$.
\end{defn}

This topological space is
quasi-compact, in the sense that every open cover admits a finite subcover, see
\cite[Theorem~40]{SZ} for the case $S = \Spec(k)$, and \cite[Prop.3.1.10]{Tem11} for general $S$.

Note that the Riemann-Zariski space is functorial only for dominant morphisms.
Our category $\loc(X)$ above is the functorial version: it is the union of
the Riemann-Zariski-spaces of all integral $X$-schemes of finite type. 



\section{Presheaves on categories of valuation rings}\label{sec:presheaves}

\subsection{Generalitites}

We now introduce our main player.
We fix a base scheme $S$.

\begin{defn}Let $X/S$ be of finite type. Let $\Fh$ be a presheaf on
one of the categories of local rings $\loc(X)$ of Definition~\ref{defn:val} over $X$.

We define $\Fh_\loc(X)$
as a global section of the presheaf $\Spec(R) \mapsto \Fh(\Spec(R))$
on $\loc(X)$, i.e., as the projective limit
\[ \Fh_\loc(X)=\varprojlim_{\loc(X)}\Fh(\Spec(R))\]
over the respective categories. 
\end{defn}

\begin{rem} \label{rem:altDescription}
This means that an element of $\Fh_\loc(X)$ is defined 
as a system of elements $s_R\in\Fh(\Spec(R))$ indexed by objects
$\Spec(R) {\to} X \in \loc(X)$ which are compatible in the sense that for every morphism $\Spec(R) \to \Spec(R')$ in $\loc(X)$, we have $s_{R'}|_R = s_R$. 
Note that this is an abuse of notation, since the element $s_R$ does \emph{not only} depend
on $R$ but also on the structure map $\Spec(R)\to X$. Most of the time the
structure map will be clear from the context. 
\end{rem}

\begin{rema} \label{rema:altDesc}
We have the following equivalent definitions of $\Fh_\loc(X)$.
\begin{enumerate}
 \item $\Fh_\loc(X)$ is the equaliser of the canonical maps 
 \[ 
\Fh(X) = \eq\left ( \prod_{\substack{\Spec(R){\to}X\\ \in \loc(X)}} \Fh(R) 
\underset{d^1}{\stackrel{d^0}{\rightrightarrows}} 
\prod_{\substack{\Spec(R){\to}\Spec(R'){\to}X\\ \in \loc(X)}} \Fh(R) \right ) .
\]
In particular, since valuation rings are the ``local rings'' for the $\rh$-site, cf. \cite{GabKel}, the construction $\Fh_\val$ can be thought of as a naïve Godement sheafification (it differs in general from the Godement sheafification because colimits do not commute with infinite products).

 \item $\Fh_\loc$ is the (restriction to $\Schft[S]$ of the) right Kan extension along the inclusion $\iota: \loc \subseteq \Sch[S]$.
 \[ \Fh_\val = (\iota^!\Fh)|_{\Schft[S]}. \]

 \item \label{rema:altDesc:nattran} $\Fh_\loc(X)$ is the set of natural transformations
\[ \Fh_\loc(X) = \hom_{\PreShv(\loc)}(h^X, \Fh) \]
where $h^X = \hom_{\Sch[S]}(-, X)$.
\end{enumerate}
\end{rema}

\begin{lemma} Let $\Fh$ be a presheaf on $\loc$. Then
the assigment $X\mapsto \Fh_\loc(X)$ defines a presheaf $\Fh_\loc$ on $\Schft[S]$.
\end{lemma}
\begin{proof}Consider a morphism $f:X\to Y$ of schemes of finite type
over $k$. Composition of $\Spec(R) \to X$ with $X \to Y$ defines a functor $\loc(X)\to \loc(Y)$ and hence a homomorphism of limits $f^*:\Fh_\loc(Y)\to\Fh_\loc(X)$. 
\end{proof}

\begin{rem} \label{rem:comparison}
We are going to show in Proposition~\ref{prop:omega_equal} that for $S=\Spec(k)$
with $k$ a perfect field, we have
\[ \Omega^n_\val=\Omega^n_\valbig=\Omega^n_{\valone}=\Omega^n_\shval.\]
In \cite{HKK}, we systematically studied the case of the category $\dvr$.
If every $X \in \Schft$ admits a proper birational morphism from a smooth $k$-scheme, we also have
\[ \Omega^n_\val=\Omega^n_\dvr \]
because both are equal to $\Omega^n_\cdh$ in this case. 
In positive characteristic, the only cases that we know $\Omega^n_\val(X) = \Omega^n_\dvr(X)$
unconditionally are if either $n=0$ (see Remark~\ref{rem:all_n=0}), $n=\dim X$ (see Proposition~\ref{prop:sdh_top}), $\dim X<n$ (in which case both are zero) or
if $\dim X \leq 3$.
\end{rem}

\begin{lemma}\label{lem:compare_big}
Let $\Fh$ be a presheaf on $\valbig$. Then $\Fh_\val=\Fh_\valbig$
as presheaves on $\Schft[S]$. 
A section over $X$ is uniquely determined by the value on the residue fields of $X$ and their valuation rings, that is, the maps $\Fh_\loc(X) \to \prod_{\rval(X)} \Fh(R)$ are injective for $\loc = \val, \valbig$ and all $X \in \Schft[S]$.
\end{lemma}

\begin{proof}
We begin with the second statement. Suppose $s, t \in \Fh_\loc(X)$ are sections such that $s|_R = t|_R$ for all $R \in \rval(X)$. If $R\subseteq K$ is valuation ring, then by Proposition~\ref{valuations},
the intersection $R'=R\cap \kappa(x)\subseteq K$ is a valuation ring. 
Then the sections $s_R, t_R$ for
$\Spec(R)\to X$, cf. Remark~\ref{rem:altDescription}, agree by the compatibility condition $s_R = s_{R'}|_R = t_{R'}|_R = t_R$.

Now the first statement. Since we have a factorisation 
\[ \Fh_\valbig(X) \to \Fh_\val(X) \to \prod_{\rval(X)} \Fh(R),\]
 it follows that the first map is injective, and we only need to show it is surjective. %

Let $s\in\Fh_\val(X)$ be global section. Defining $t_R \stackrel{def}{=} s_{R'}|_R$, with $R$ and $R'$ as above, we get a candidate element $t = (t_R) \in \prod_{\valbig(X)} \Fh(R)$ which is potentially in $\Fh_\valbig(X)$. Let $\Spec(R_1){\to}\Spec(R_2)$ be a morphism in $\valbig(X)$. Let $R_1', R_2'$ be the valuation rings of residue fields of $X$ corresponding to $R_1, R_2$ as above, but since there is not necessarily a morphism $\Spec(R_1'){\to}\Spec(R_2')$, we also set $\p_2' = \ker(R_2' \to R_1)$ and $S = \kappa(\p_2') \cap R_1$, to obtain the following commutative diagram in $\valbig(X)$, with $R_1', R_2', S \in \val(X)$ by Proposition~\ref{valuations}.
\begin{equation} \label{diagramChase1}
 \xymatrix@R=0pt{
& & \Spec(R_1') \\
& \Spec(S) \ar[ur] \ar[dr] & \\
\Spec(R_1) \ar[ur] \ar[dr]  & & \Spec(R_2') \\
 & \Spec(R_2) \ar[ur]& 
} 
\end{equation}
Now the result follows from a diagram chase: We have $
t_{R_1} 
= s_{R_1'}|_{R_1} 
= s_{R_1'}|_{S}|_{R_1}
= s_{S}|_{R_1}
= s_{R_2'}|_{S}|_{R_1}
= s_{R_2'}|_{R_2}|_{R_1}
= t_{R_2}|_{R_1}
$.
\end{proof}

Recall that a presheaf on $\loc$ is torsion-free if it sends dominant morphisms to monomorphisms, see Definition~\ref{defn:torsion-free}.

\begin{lemma} \label{lemm:rsF}
Let $\loc \in \{\val,\valbig,\valone, \rval\}$.
Let $\Fh$ be a torsion-free presheaf on  $\loc(X)$. Then $\Fh_\loc(X)$ is canonically isomorphic to
\begin{equation} \label{equa:rsF}
\left \{ (s_x) \in \prod_{x \in X} \Fh(x) \middle | \begin{array}{c}
\textrm{ for every } T{\to}X \in \loc(X) \textrm{ there exists } \\
s_T \in \Fh(T) 
 \textrm{ such that } 
(s_x)|_{\prod_{t \in T} F(t)} = 
s_R|_{\prod_{t \in T} F(t)} 
\end{array} \right \}. 
\end{equation}
\end{lemma}

It is perhaps worth noting that the description in Equation~\eqref{equa:rsF} is basically the presheaf $\mathrm{rs\ }F$ from the proof of \cite[Prop.3.6.12]{Kel12}.


\begin{proof}By torsion-freeness, the projection $\Fh_\loc(X) \to \Fh_{\valr{0}}(X)$ is injective. By functoriality the map $\Fh_{\valr{0}}(X) \to \prod_{x \in X} \Fh(x)$ is injective.

Assume conversely we are given a system of $s_x$ as in Equation~\ref{equa:rsF}.  
As in the proof of the previous lemma, this gives us a candidate section
$t_R \in \Fh(\Spec(R))$ for all $\Spec(R)\in\loc(X)$. 
It remains to check compatibility of these sections.
Let $\Spec(R_1)\to\Spec(R_2)$ be a morphism in $\loc(X)$. The generic
point of $\Spec(R_1)$ maps to a point of $\Spec(R_2)$ corresponding
to a prime ideal $\p_2\subset R_2$. By torsion-freeness, we may
replace $R_1$ by its field of fractions and $R_2$ by $(R_2)_{\p_2}$.
In other words, $R_1$ is a field containing the residue field
of $R_2$. %
Now the same diagram chase as for Diagram~\ref{diagramChase1} works. Since $S = R_2' / \m$ and $x = \Spec(R_1')$ is the image of $\Spec(R_1)$ in $X$ we have (keeping in mind the condition of Equation~\ref{equa:rsF} above): $
t_{R_1} 
= s_{x}|_{R_1} 
= s_{x}|_{R_2'/\m}|_{R_1}
= s_{\Frac(R_2')}|_{R_2'/\m}|_{R_1}
= s_{\Frac(R_2')}|_{R_2}|_{R_1}
= t_{R_2}|_{R_1}
$.
\end{proof}


\subsection{Reduction to strictly henselian valuation rings}
The aim of this section is to establish that using strictly henselian local rings gives the same result:

\begin{prop}\label{prop:reduce_shval}Let $\Fh$ be a presheaf on $\val$ that commutes with filtered colimits and satifies the sheaf condition for the \'etale topology. 
Then the canonical projection morphism
\[ \Fh_\val(X)\to\Fh_\shval(X)\]
is an isomorphism.
\end{prop}

\begin{proof}
Let $s, t\in\Fh_\val(X)$ be sections such that the induced elements of
$\Fh_\shval(X)$ agree. Let $x\in X$ be a point with residue field $\kappa$.
The separable closure $\kappa^s$ of $\kappa$ is in the category
$\shval$. By assumption
\[ \Fh(\kappa^s)=\varinjlim_{\lambda/\kappa}\Fh(\kappa)\]
where $\lambda$ runs through the finite extensions of $\kappa$ contained in $\kappa^s$.
The vanishing of $s_{\kappa^s}$ implies that there is one such $\lambda$ with
$s_\lambda=0$. As $\lambda \subseteq \kappa^s$, the morphism $\Spec(\lambda){\to}\Spec (\kappa)$ is \'etale. As a consequence
of \'etale descent, we know that the map $\Fh(\kappa)\to\Fh(\lambda)$ is injective, hence $s_\kappa=t_\kappa$. The same argument also applies to a valuation ring
$R$ of $\kappa$ and its strict henselisation, viewed as the colimit of Equation~\eqref{equa:strictHenselisation}, cf. Corollary~\ref{coro:etaleValCover}\eqref{coro:etaleValCover:shval}.
 Hence
the morphism in the statement is injective.

Now let $t\in\Fh_\shval(X)$. For any $\Spec(R) \to X$ in $\val(X)$ with strict henselisation $R^{sh}$, let $G = \Gal(\Frac(R^{sh})  / \Frac(R))$. Since $\Frac(R) = \Frac(R^{sh})^G$, and $R = R^{sh} \cap \Frac(R)$, we have $R = (R^{sh})^{G}$. In particular, the element $t_{R^{sh}}$ lifts to $\Fh(R) \subseteq \Fh(R^{sh})$, as $t_{R^{sh}}$ must be compatible with every morphism $\Spec(R^{sh}) \to \Spec(R^{sh})$ in $\shval(X)$. In this way, we obtain an element $t_R$ for every $\Spec(R) \to X$ in $\val(X)$. We want to know that these form a section of $\Fh_{\val}(X)$. But compatibility with morphisms $\Spec(R) \to \Spec(R')$ of $\val(X)$ follows from the definition of the $t_R$, the fact that strict henselisations are functorial \cite[Tag 08HR]{stacks-project}, and the morphisms $\Fh(R) \to \Fh(R^{sh})$ being injective, which we just proved.
%
\end{proof}

\subsection{Reduction to rank one}

The aim of this section is to establish that using rank one valuation rings gives the same result:

\begin{prop}\label{prop:reduce_rk_one}Let $\Fh$ be a torsion-free presheaf on $\val$. Assume that for every valuation
ring $R$ in $\val$ and prime ideal  $\p$ the diagram
\[ \xymatrix{
\Fh(R) \ar@{^{(}->}[r] \ar[d] & \Fh(R_\p) \ar[d] \\
\Fh(R / \p) \ar@{^{(}->}[r] & \Fh(\Frac(R / \p))
} \]
is cartesian.
Then the natural restriction
\[ \Fh_\val \to\Fh_{\valone}\]
is an isomorphism.
\end{prop}

\begin{proof}
Recall that $\val^{\leq r}$ is defined to be the subcategory of $\val$ using
only valuations of rank at most~$r$, and $\Fh_{\val^{\leq r}}$ denotes the presheaf obtained using only 
$\val^{\leq r}$.
There are canonical morphisms
\[ \Fh_{\val} \to \Fh_{\val^{\leq r}} \to \Fh_{\val^{\leq r-1}} \]
for all $r \geq 1$. By torsion free-ness these are all subpresheafs of $\Fh_{\val^{\leq 0}}$, and so the two morphisms above are monomorphisms.
Moreover, $\Fh_\val=\bigcap_r\Fh_{\val^{\leq r}}$ because by definition, all fields in
$\val$ are of finite transcendence degree over $k$ and hence all valuation rings
have finite rank, see Proposition~\ref{valuations}.
 Hence it suffices
 to show that $\Fh_{\leq r}\to\Fh_{\leq r-1}$ is an epimorphism for all $r$.

Let $\p$ be a prime ideal of a valuation ring $R$.  
By Corollary~\ref{coro:is_valuation} both $R_\p$ and $R / \p$ are valuation rings and if $\p$ is not maximal or zero, then $R_\p$ and $R / \p$ are of rank smaller than $R$. Indeed, $\rank R = \rank R_\p + \rank R / \p$ since the rank is equal to the Krull dimension and the set of ideals, and in particular prime ideals, of a valuation ring is totally ordered, see Proposition~\ref{prop:val_ordered}. 

Let $t \in \Fh_{\val^{\leq r-1}}(X)$ be a section. When $r \geq 2$, we can choose a canonical candidate $s \in \prod_{R\in\val^{\leq r}} \Fh(\Spec(R))$ for an element of $\Fh_{\val^{\leq r}}(X)$ in the preimage of $t$: for every valuation ring of rank $r$, take $\p$ to be any prime ideal with $R_\p, R/\p \in \val^{\leq r-1}$ and construct a section over $R$ using the cartesian square of the assumption.
 Since the morphisms $\Fh(\Spec(R)) \to \Fh(\Spec(R_\p))$ are monomorphisms, the choice of $\p$ does not matter. 

It remains to check that this candidate section $s \in \prod_{\val^{\leq r}} \Fh(R)$  is actually a section of $\Fh_{\val^{\leq r}}(X)$. I.e., for any $X$-morphism of valuation rings $\Spec(R') \to \Spec(R)$ with $R$ or $R'$ (or both) of rank $r$, we want to know that the element $s_R$ restricts to $s_{R'}$. The morphism $\Fh(R') \to \Fh(\Frac(R'))$ is injective, so it suffices to consider the case when $R'$ is some field, $L$. Then $R \to L$ factors as
\[ R \to R / \p \to \Frac(R / \p) \to L \]
where $\p$ is the prime $\p = \Ker(R \to L)$.  That $s_R$ is sent to $s_{R / \p}$ comes from the independence of the choice of $\p$ that we used to construct $s$. For the same reason, $s_{R / \p}$ is sent to $s_{\Frac(R / \p)}$. Finally, both $\Frac(R / \p)$ and $L$, being fields, are of rank zero, and so $s_{\Frac(R / \p)} = t_{\Frac(R / \p)}$ is sent to $s_{L} = t_L$ since $t$ is already a section of $\Fh_{\val^{\leq r-1}}$.
\end{proof}

\begin{coro}
For any scheme $Y$, writing $h^Y$ for the presheaf $\hom_{\Sch}(-, Y)$, the canonical maps are isomorphisms
\[ h^Y_\valbig \cong h^Y_\val \cong h^Y_{\valone} \cong h^Y_\shval. \]
\end{coro}

\begin{proof}
The first isomorphism is Corollary~\ref{lem:compare_big}. The second one follows from Proposition~\ref{prop:reduce_rk_one}---note that $\Spec(R) = \Spec(R/\p) \amalg_{\Spec(\Frac(R/\p))} \Spec(R_\p)$ by \cite[Thm.5.1]{Fer03}. The last one follows from Proposition~\ref{prop:reduce_shval}.
\end{proof}

\subsection{The case of differential forms}

Now we show that the previous material applies to the case of differential forms:

\begin{prop}\label{prop:omega_equal}
Let $S=\Spec(k)$ with $k$ perfect field.
The canonical morphisms
\[ \begin{xy}\xymatrix{
&&\Omega^n_\shval\\
\Omega^n_\valbig \ar@/^3ex/[rru] \ar@/_3ex/[rrd] \ar[r]_{\cong}^{\textrm{\footnotesize{Lem{.}\ref{lem:compare_big}}}} &
\Omega^n_\val \ar[ru]^\cong_{{Prop.\ref{prop:reduce_shval}}} \ar[rd]_\cong^{Prop.\ref{prop:reduce_rk_one}} &\\
&&\Omega^n_{\valone}
}\end{xy}\]
are isomorphisms of presheaves on $\Schft$.
\end{prop}
\begin{proof}
The presheaf $\Omega^n$ on $\valbig$ is an étale sheaf and commutes with direct limits.
Hence we may apply Proposition~\ref{prop:reduce_shval} in order to show the comparison to $\Omega^n_\shval$.

For the final isomorphism we want to apply Proposition~\ref{prop:reduce_rk_one}. The rest of this section is devoted to checking the necessary cartesian diagram. 
\end{proof}

\begin{lemma} \label{lemm:openclosedGlueingValRings}
Let $R$ be a valuation ring, $\p$ a prime ideal.
Then the diagram
\[ \xymatrix{
R \ar@{^{(}->}[r] \ar[d] & R_\p \ar[d] \\
R / \p \ar@{^{(}->}[r] & R_\p / \p R_\p \cong \Frac(R / \p)
} \]
is cartesian and the canonical $R$-module morphism $\p \to \p R_\p$ is an isomorphism.
\end{lemma}
\begin{proof}
We have to check that an element $a/s$ of $R_\p$ is in $R$ if its reduction modulo
$\p R_\p$ is in $R/\p$. This amounts to showing that $\p=\p R_\p$: if there is $b \in R$ which agrees with $a / s$ mod $\p R_\p$, then $b - a/s \in \p R_\p$. But if $\p=\p R_\p$, then $b - a/s \in \p \subset R$, i.e., $a/s \in R \subseteq R_\p$. 

Let $x=a/s\in\p R_\p$ with $a\in \p$ and $x\in R-\p$. Let $v$ be the valution of $R$. We compare
$v(a)$ and $v(s)$. If we had $v(s)\geq v(a)$, then $x^{-1}\in R$ and hence
$s=x^{-1}a\in \p$ because $\p$ is an ideal. This is a contradiction. Hence $v(s)<v(a)$ and
$x\in R$. Now we have the equation $sx=a$ in $R$, hence $sx\in \p$. As
$\p$ is a prime ideal and $s\notin \p$, this implies $x\in \p$.
\end{proof}

\begin{lemma} Let $R$ be a $k$-valuation ring, $\p$ a prime ideal. Then the
diagram
\[ \xymatrix{
\Omega^1(R) \ar@{^{(}->}[r] \ar@{->>}[d] & \Omega^1(R_\p) \ar@{->>}[d] \\
\Omega^1(R / \p) \ar@{^{(}->}[r] & \Omega^1(\Frac(R / \p))
} \]
is cartesian.
\end{lemma}
\begin{proof} The $R$-module $\Omega^1(R)$ is flat because it is torsion-free over a valuation ring. 
We tensor the diagram of the last lemma with the flat $R$-module $\Omega^1(R)$ 
and obtain the cartesian diagram
\[\xymatrix{
\Omega^1(R) \ar@{^{(}->}[r] \ar[d] & \Omega^1(R_\p) \ar[d] \\
\Omega^1(R)\tensor_R R / \p \ar@{^{(}->}[r] & \Omega^1(R_\p)\tensor_{R_\p}R_\p / \p R_\p
} \]
In the next step we use the fundamental exact sequence for differentials of
a quotient \cite[Theorem~58, p.~187]{Mat} and obtain the following diagram with exact columns. The top horizontal arrow is an isomorphism, and in particular a surjection, because $\p=\p R_\p$. 
\begin{equation} \label{equa:diagramChaseSESmonic}
\xymatrix{
\p/\p^2\ar[r]^{\textrm{epi}} \ar[d]& \p R_\p/\p^2R_\p\ar[d]\\
\Omega^1(R)\tensor_R R / \p \ar@{^{(}->}[r]^-{\textrm{monic}} \ar@{->>}[d]_{\textrm{epi}} & \Omega^1(R_\p)\tensor_{R_\p}R_\p / \p R_\p\ar[d]\\
\Omega^1(R/\p)\ar@{^{(}->}[r]\ar[d] &\Omega^1(R_\p/\p R_\p)\ar[d]\\
0&0
} 
\end{equation}
A small
diagram chase now shows that the second square is cartesian. Putting the two
diagrams together, we get the claim.
\end{proof}

\begin{lemma}\label{lem:omega_cartesian}
Let $R$ be a $k$-valuation ring, $\p$ a prime ideal. Then the diagram
\[ \xymatrix{
\Omega^n(R) \ar[r] \ar[d] & \Omega^n(R_\p) \ar[d] \\
\Omega^n(R / \p) \ar[r] & \Omega^n(\Frac(R / \p))
} \]
is cartesian for all $n\geq 0$.
\end{lemma}

\begin{proof}We have already done the cases $n=0,1$. 
We want to go from $n=1$ to general $n$ by taking exterior powers. 
We write $\Omega^1(R)$ as the union of its finitely generated sub-$R$-modules
\[ \Omega^1(R)=\bigcup N. \]
We write $\overline{N}$ for the image of $N$ in $\Omega^1(R/\p)$. Note
that
\begin{gather*} 
\Omega^1(R/\p)=\bigcup \overline{N}\\
\Omega^1(R_\p)=\Omega^1(R)_\p=\bigcup N_\p\\ 
\Omega^1(R_\p/\p )=\Omega^1(R/\p)_\p =\bigcup \overline{N}_\p
\end{gather*}

The module $N$ is a torsion free finitely generated module over the valuation
ring $R$, hence free. The $R_\p, R / \p, R_\p / \p$ modules $N_\p, N / \p, N_\p/ \p$ are therefore also free, and the same is true for all the exterior powers. So for all $n \geq 0$ we get
the following cartesian diagram: 
\[\xymatrix{
\bigwedge^n N\ar[d]\ar[r]&\bigwedge^nN_\p\ar[d]\\
\bigwedge^n(N/\p )\ar[r]&\bigwedge^n (N_\p/\p )
}\]
Note that if the rank of $N$ is one, and $n = 0$, this is just the cartesian diagram from Lemma~\ref{lemm:openclosedGlueingValRings}. Note also that at this stage we are working with $N / \p$ and $N_\p / \p $, instead of $\overline{N}$ and $\overline{N}_\p$.

Passing to the direct limit, we have established as a first step that the
diagram
\[\xymatrix{
\Omega^n(R)\ar[d]\ar[r]&\Omega^n(R_\p)\ar[d]\\
\Omega^n(R)/\p \ar[r]&\Omega^n(R_\p) / \p 
}\]
is cartesian.

Let $\pi:\Omega^n(R)/\p\to \Omega^n(R/\p)$ and $\theta:\Omega^n(R_\p)/\p \to \Omega^n(R_\p/\p R_\p)$ be the natural maps.  We want to show that 
\[\xymatrix{
\Omega^n(R)/\p\ar[d]\ar[r]&\Omega^n(R_\p)/\p\ar[d]\\
\Omega^n(R/\p) \ar[r]&\Omega^n(R_\p/\p) 
}\]
is also cartesian.
Let $\pi^{-1}\overline{N} \subseteq \Omega^1(R) / \p$ be the preimage of
$\overline{N}$. Similarly, let $\theta^{-1}\overline{N}_\p \subseteq \Omega^1(R_\p) / \p$ be the preimage of $\overline{N}_\p$ by $\theta: \Omega^1(R_\p)/\p \to \Omega^1(R_\p/\p)$. In particular, we have the following cube, for which the two side squares are cartesian by definition, the front square is the cartesian square from Diagram~\eqref{equa:diagramChaseSESmonic} on page~\pageref{equa:diagramChaseSESmonic}, and a diagram chase then shows that the back square is also cartesian. Note, that the lower and upper faces are probably not cartesian, but this does not affect the argument.
\[\xymatrix@!=12pt{
\pi^{-1}\overline{N} \ar@{^{(}->}[rr] \ar@{->>}[dd]_\pi \ar@{^{(}->}[dr] && \theta^{-1}\overline{N}_\p \ar@{->>}[dd]_(0.3){\theta}|\hole \ar@{^{(}->}[dr] \\
&\Omega^1(R)/\p \ar@{^{(}->}[rr] \ar@{->>}[dd] && \Omega^1(R_\p) / \p \ar@{->>}[dd] \\
\overline{N}\ar@{^{(}->}[rr]|\hole \ar@{^{(}->}[dr] && \overline{N}_\p \ar@{^{(}->}[dr] \\
&\Omega^1(R/\p) \ar@{^{(}->}[rr] && \Omega^1(R_\p / \p)
}\]
We claim that the back square stays cartesian when passing to higher exterior powers.
We will do this by comparing the kernels of $\bigwedge^n\pi$ and
$\bigwedge^n \theta$, cf. the diagram chase of Diagram~\eqref{equa:diagramChaseSESmonic} on page~\pageref{equa:diagramChaseSESmonic}. More precisely, to show that the higher exterior powers of the back square are cartesian, it suffices to show that $\Ker \bigwedge^n\pi \to \Ker \bigwedge^n\theta$ is a surjection. We will show that it is an isomorphism.

Note that since the back face is cartesian, and the horizontal morphism is a monomorphism, we have $\ker(\pi) \cong \ker(\theta)$. Let $X$ be this common kernel.
The module $\overline{N}\subset\Omega^1(R / \p)$ is torsion free and finitely generated over the valuation ring $R / \p$, and therefore it is free, and in particular, projective. Hence $\pi$ admits  a splitting $\sigma$.

The map $\iota = \sigma \otimes_{R / \p} R_\p / \p$ is then
a splitting of $\theta$ compatible with $\sigma$.
In particular, we have compatible decompositions
\[ \pi^{-1}\overline{N}  =X\oplus \overline{N}, \qquad \theta^{-1}\overline{N}_\p =X \oplus \overline{N}_\p.\]
The $X$'s are the same due to the square being cartesian. 
Hence
\begin{gather*}
\bigwedge_{R / \p}^n\pi^{-1}\overline{N} =\bigoplus_{i=0}^n\bigwedge_{R / \p}^i X \underset{R / \p}{\otimes} \bigwedge_{R / \p}^{n-i}\overline{N} \qquad  
\\
\bigwedge^n_{R_\p / \p}\theta^{-1}\overline{N}_\p =\bigoplus_{i=0}^n\bigwedge_{R_\p / \p}^i X \underset{R_\p / \p}{\otimes} \bigwedge_{R_\p / \p}^{n-i}\overline{N}_\p  \qquad  
\end{gather*}
The kernels of $\bigwedge^n\pi$ and $\bigwedge^n\pi_F$ are given by the analogous  sums but indexed by $i = 1, \dots, n$. Note that $X$ is an $R_\p / \p = \Frac(R / \p)$-vector space.
Since $R_\p / \p = \Frac(R / \p)$, if $A,B$ are $R_\p / \p$-vector spaces, then
\[ A\tensor_{R / \p} B=A\tensor_{R_\p / \p} B\]
hence
\[\bigwedge^i_{R / \p}X=\bigwedge^i_{R_\p / \p}X\]
and finally
\[ \left ( \bigwedge^i_SX \right ) \tensor_S \overline{N} 
= \left ( \bigwedge^i_SX \right ) \tensor_F F \tensor_S \overline{N} 
= \left ( \bigwedge^i_FX \right ) \tensor_F \overline{N}_\p.\]
where $S = R / \p$ and $F = \Frac(R / \p) = R_\p / \p$. There are similar formulas for higher exterior powers of $\overline{N}_\p$ and $\overline{N}$. 
So we have shown that $\ker \bigwedge^n\pi \to \ker \bigwedge^n\theta$ is an isomorphism as claimed.
\end{proof}


\begin{rem}This finishes the proof of Proposition~\ref{prop:omega_equal}.
\end{rem}


\section{Descent properties of $\Omega^n_\val$}\label{sec:descent}

Our presheaf of interest, the presheaf $\Omega^n$, is a sheaf for the \'etale topology. This has far reaching consequences. 

\begin{rema} \label{rema:hYalso}
We signal that we have written $\Omega^n$ everywhere because this is our main object of study, but everything in this section is valid for any presheaf $\Fh$ on $\langle \val(S), \Schft[S] \rangle$, (the full subcategory of $S$-schemes whose objects are those of $\val(S)$ and $\Schft[S]$), satisfying:
\begin{enumerate}
 \item[(Co)] $\Fh$ commutes with filtered colimits.
 \item[(Et)] $\Fh$ satifies the sheaf condition for the \'etale topology.
\end{enumerate}
For example, if $Y$ is any scheme, then $\Fh(-) = h^Y(-) = \hom(-, Y)$ satisfies these conditions.
\end{rema}

\begin{prop}
Suppose that $\Fh$ is a presheaf on $\shval(X)$. Then $\Fh_\shval$ is an $\eh$-sheaf. Similarly, $\Fh_\val$ is an $\rh$-sheaf for any presheaf $\Fh$ on $\val(X)$.
\end{prop}

\begin{rem}
If we had defined a category $\hval$ of hensel valuation rings, we could also have said that $\Fh_\hval$ is a $\cdh$-sheaf for any presheaf $\Fh$ on $\hval(X)$.
\end{rem}

\begin{proof}
We only give the $\shval(X), \Fh_\shval, \eh$ proof, as the same proof works for $\val(X), \Fh_\val, \rh$. 
Let $U \to X$ be an $\eh$-cover. The map $\Fh_\shval(X) \to \Fh_\shval(U)$ is injective because by Proposition~\ref{prop:eh_val} every morphism $\Spec(R) \to X$ from a strictly henselian valuation ring factors through $U$.

Now suppose that $s \in \Fh_\shval(U)$ satisfies the sheaf condition for the cover $U\to X$. Let $f:\Spec(R)\to X$ be in $\shval(X)$. By choosing a lifting $g:\Spec(R) \to U \to X$, we obtain a candidate section $s_g\in\Fh(\Spec(R))$. We claim that it is independent of the choice of lift. Let $g'$ be a second lift. The pair
$(g,g')$ defines a morphism $\Spec(R)\to U\times_XU$. By assumption,
$s$ is in the kernel of  $\Fh_\shval(U) \xrightarrow{pr_1 - pr_2} \Fh_\shval(U \times_X U)$. In particular $s_g=s_{g'}\in\Fh(\Spec(R))$ in the $(g,g')$-component.
Let $f_1:\Spec(R_1)\to X$ and $f_2:\Spec(R_2)\to X$ be in $\loc(X)$ and
$\Spec(R_1)\to\Spec(R_2)$ an $X$-morphism. The choice
of a lifting $g_2:\Spec(R_2)\to U$ also induces a lifting $g_1$.
The section $s_{g_2}$ restricts to $s_{g_1}$, hence $s_{f_2}$ restricts to
$s_{f_1}$. Our candidate components define an element of $\Fh_\val(X)$.
\end{proof}

\begin{coro}Let $k$ be a perfect field.
The presheaf $\Omega^n_\val$ is an $\eh$-sheaf on $\Schft$. More generally, $\Fh_\val$ is an $\eh$-sheaf on $\Schft[S]$ for any presheaf $\Fh$ on $\val(S)$ which commutes with filtered colimits and satifies the sheaf condition for the \'etale topology.
\end{coro}

\begin{proof}
Follows from $\Omega^n_\val = \Omega^n_\shval$, see Proposition~\ref{prop:reduce_shval}.
\end{proof}

This immediately implies:

\begin{cor} \label{cor:injective}
The map of presheaves
$\Omega^n\to\Omega^n_\val$ on $\Schft$ induces maps of presheaves
\begin{equation} \label{canonicalMapEhVal}
\Omega^n_\rh \to \Omega^n_\cdh \to \Omega^n_\eh\to\Omega^n_\val.
\end{equation}
More generally, this is true for any presheaf $\Fh$ as in Remark~\ref{rema:hYalso}.
\end{cor}

We find it worthwhile to restate the following theorem from \cite{HKK}. The original statement is for $\Omega^n$ and $S = \Spec(k)$, but one checks directly that the proof works for any Zariski sheaf and the relative Riemann-Zariski space, Definition~\ref{defn:RZ}, of \cite{Tem11}.

\begin{thm}[{\cite[Theorem~A.3]{HKK}}] \label{theo:A3}
For any presheaf $\Fh$ on $\Schft[S]$ the following are equivalent.
\begin{enumerate}
 \item \label{theo:A3:1} cf.\cite[Hyp.H]{HKK} For every integral $X \in \Schft[S]$ and $s, t \in \Fh(X)$ such that $s|_U = t|_U$ for some dense open $U \subseteq X$, there exists a proper birational morphism $X' \to X$ with $s|_{X'} = t|_{X'}$.
 \item cf.\cite[Hyp.V]{HKK} For any $R \in \val(S)$ with fraction field $K$ the morphism $\Fh'(R) \to \Fh'(K)$ is injective. That is, $\Fh'$ is torsion-free, Definition~\ref{defn:torsion-free}.
\end{enumerate}
Here $\Fh'$ is the presheaf sending $R \in \val(S)$ to the colimit $\Fh'(R) = \varinjlim_{(R \downarrow \Schft[S])} \Fh(Y)$ over factorisations $\Spec(R) \to Y \to S$ through $Y \in \Schft[S]$. Clearly, $(\Omega^n)' = \Omega^n$, and more generally, for any presheaf $\Fh$ as in Remark~\ref{rema:hYalso} one has $\Fh' = \Fh$.
\end{thm}

\begin{cor}\label{cor:injective}
The maps $\Omega^n_\rh \to \Omega^n_\cdh \to \Omega^n_\eh\to\Omega^n_\val$ are injective. More generally, the same is true for any presheaf $\Fh$ as in Remark~\ref{rema:hYalso}.
\end{cor}

\begin{proof} 
Cf. \cite[Cor.5.10, Prop.5.12]{HKK}. It suffices to show that for all any $X \in \Schft[S]$ and any $\omega, \omega' \in \Omega^n(X)$ with $\omega|_{\prod_{x \in X} \Omega^n(x)} = \omega'|_{\prod_{x \in X} \Omega^n(x)}$ there is a $\cdp$-morphism $X' {\to}X$ with $\omega|_{X'} = \omega'|_{X'}$. By noetherian induction, it suffices to show that Theorem~\ref{theo:A3}\eqref{theo:A3:1} is satisfied. But by Theorem~\ref{theo:A3}, this is equivalent to being torsion-free.
%
%
\end{proof}
In order to prove surjectivity, we need a strong compactness property.

\begin{lemma} \label{lemm:valRingConstOne}
Let $(Y_i)_{i\in I}$ be a filtered system of noetherian topological spaces. 
Then there is a system of irreducible closed subsets $Z_i \subseteq Y_i$
such that $\overline{\pi_{ij}(Z_i)} = Z_j$.
\end{lemma}

\begin{proof}
To every $i$ we attach a finite set $F_i$ as follows:
Let $V_1,\dots,V_n$ be the irreducible components of $Y_i$. Let
$F_i$ be the set of the irreducible components of all multiple intersections
$V_{m_1}\cap\dots\cap V_{m_k}$ for all $1\leq k\leq n$ and all choices
of $m_j$. We define a transition map $\pi_{ij}$ by mapping an element 
$y\in F_{i}$ to the smallest element of $F_{j}$ containing its image. This defines a filtered system of non-empty finite sets. Its projective limit is non-empty by \cite[Tag 086J]{stacks-project}. Let $(W_i)_{i\in I}$ be an element of the limit.

Now for each $j \in I$, consider the partially ordered set of closures of images $\{ \overline{\pi_{ij}(W_i)} \in Y_j : {i \leq j} \}$. We claim that there is an $i_j$ such that $\overline{\pi_{i'j}(W_{i'})} = \overline{\pi_{i_jj}(W_{i_j})}$ for all $i' \leq i_j$. Indeed, if not, then we can construct a strictly decreasing sequence of closed subsets of $Y_j$, contradicting the fact that it is a noetherian topological space. Define $Z_j = \overline{\pi_{i_jj}(W_{i_j})} \subseteq Y_j$ for such an $i_j$. It follows from our definitions of the $Z$ that for every $i, j \in I$ with $i \leq j$, we have $\overline{\pi_{ij}(Z_i)} = Z_j$.
\end{proof}

Let $X$ be integral and
$Y\to X$ proper birational. Fix $\omega, \omega' \in \Omega^n_\val(X)$ (or in $\Fh_\val(X)$ if we are using an $\Fh$ as in Remark~\ref{rema:hYalso}) and define $Y^{\omega \neq \omega'} \subset Y$ to be the subset of points $y \in Y$ for which $\omega_y \neq \omega'_y$. We view it as a topological space with the induced topology. 
\begin{lemma}\label{lem:topology}
The topological space $Y^{\omega \neq \omega'}$ is noetherian. 
In addition, every irreducible closed subset of $Y^{\omega \neq \omega'}$ has a unique generic point.
\end{lemma}
\begin{proof}The topological space of $Y$ is noetherian. Subspaces of noetherian topological spaces with the induced topology are noetherian (easy exercise), hence $Y^{\omega \neq \omega'}$ is noetherian. 
Let $Z^{\omega \neq \omega'}\subseteq Y^{\omega \neq \omega'}$ be irreducible. Note its
closure $Z\subseteq Y$ is also irreducible. Let $\eta$ be the generic point
of $Z$ in the scheme $Y$. Assume $\eta\notin Z^{\omega \neq \omega'}$. By definition this means $\omega_\eta = \omega'_\eta$. Hence, since $\Omega^n$ commutes with filtered colimits, $\omega = \omega'$ also on some open dense $U \subseteq Z$ of the reduction of $Z$. As $Z^{\omega \neq \omega'}$ is dense in $Z$, the intersection $U \cap Z^{\omega \neq \omega'}$ is non-empty implying the existence of a point $y \in Y^{\omega \neq \omega'}$ for which $\omega_y = \omega'_y$, and contradicting the definition of $Y^{\omega \neq \omega'}$. Hence $\eta\in Z^{\omega \neq \omega'}$ and we have found a generic point. Uniqueness follows easily from uniqueness of generic points in $Y$. 
\end{proof}

We are implicitly using the Riemann-Zariski space of $X$, see Definition~\ref{defn:RZ}, in the following proof.

\begin{lemma} \label{lem:kill_torsion}
Let $X$ be integral, $\omega, \omega' \in\Omega^n_\val(X)$ such that $\omega_R = \omega'_{R}$ for $X$-valuation rings $R$ of $k(X)$. Then there exists a proper birational $Y\to X$ such that
$\omega|_Y= \omega'|_Y \in\Omega^n_\val(Y)$. More generally, this is true for a presheaf $\Fh$ as in Remark~\ref{rema:hYalso}.
\end{lemma}

\begin{proof}
Suppose the contrary---that $Y^{\omega \neq \omega'}$ is non-empty for all proper birational maps
$Y\to X$.
The $Y$'s and hence also their subspaces $Y^{\omega \neq \omega'}$  form a filtered system of
noetherian topological spaces. 
By Lemma~\ref{lemm:valRingConstOne}, if all $Y^{\omega \neq \omega'}$ are non-empty, then there exists a strictly compatible system of irreducible subspaces of the $Y^{\omega \neq \omega'}$'s. Their generic points (which exist and are unique by Lemma~\ref{lem:topology}) define
a point 
$y = (y_i) \in \varprojlim Y$ for which $\omega_{k(y_i)} \neq \omega'_{k(y_i)}$ for all $y_i$. 
But $\Omega^n$ commutes with cofiltered limits of rings, so $\omega|_{k(y)} \neq \omega'|_{k(y)}$ where $k(y)$ is the residue field of the valuation ring $\varinjlim \OO_{Y_i, y_i}$ in the locally ringed space 
$\varprojlim Y$. But this is an $X$-valuation ring $R$ of $k(X)$ with
 $\omega|_{R/\m} \neq \omega'|_{R/\m}$. This contradicts the assumption that $\omega_R = \omega'_R$.
\end{proof}

\begin{prop} \label{lemm:rhValEpi}
The map $\Omega^n_\rh\to\Omega^n_\val$ is surjective. More generally, $\Fh_\rh\to\Fh_\val$ is surjective for any presheaf $\Fh$ as in Remark~\ref{rema:hYalso}.
\end{prop}

\begin{proof}
As both are $\rh$-sheaves, we may work $\rh$-locally. In particular, without
loss of generality $X$ is integral with function field $K$.
Choose $\omega=(\omega_R)_R\in\Omega^n_\val(X)$. 

We start with an $X$-valuation ring $R\subset K$, i.e., $R\in\RZ(X)$. The form $\omega_R$ is already
defined on some ring $A_R$ of finite type over $X$. Let
$\widetilde{\omega}_R$ be this class.
For all $R'$ in the Zariski-open subset of $\RZ(X)$ defined by $A_R$ (i.e.,
$A_R\subset R'$), we have $\omega_{R'}=\widetilde{\omega}_R$ because
$\Omega^n(R')$ is torsion free and the equality holds in
$\Omega^n(K)$.

As $\RZ(X)$ is quasi-compact, we may cover it by finitely many
of these and obtain a finite set of rings $A_i\subset K$ which
are of finite type over $X$. On each of these we are given
a differential form $\omega_{A_i} \in \Omega^n(A_i)$ inducing all $\omega_R$ for
$R\in\RZ(X)$ containing $A_i$. The forms $\omega|_{A_i}$ and $\omega_{A_i}$ in $\Omega^n_\val(\Spec(A_i))$ agree in $\Omega^n_\val(K)$. By Lemma~\ref{lem:kill_torsion} there exists
a blow-up $W_i\to \Spec(A_i)$ such that $\omega_{A_i}|_{W_i} = \omega|_{W_i}$.
Hence, $\omega|_{W_i}$ is represented by a class $\omega_{W_i}$ in
$\Omega^n_\rh(W_i)\subset\Omega^n_\val(W_i)$.

By Nagata compactification, there is a factorization 
\[W_i \to \bar{W}_i \to X\]
where the first map is a dense open immersion and the second is proper and birational. 
Let $V$ be the closure of $\Spec(K)$ in $\bar{W}_1\times_X\dots\times_X \bar{W}_n$. The canonical morphism $V \to X$ is proper and birational. Every point of $V$ is dominated by a valuation ring of $K$, \cite[7.1.7]{EGAII}. There is necessarily at least one of the $A_i$ contained in it, and so the base changes $V_i = W_i \times_{\bar{W}_i} V$ form an open cover of $V$. 
Let $\omega_{V_i}\in\Omega^n_\rh(V_i)$ be the restriction of $\omega_{W_i}$.

We have
$\omega_{V_i}=\omega_{V_j}$ in $\Omega^n_\val(V_i\cap V_j)$ (and hence $\Omega^n_\rh(V_i\cap V_j)$, Corollary~\ref{cor:injective}) because both agree with the restriction of $\omega$. By Zariski-descent, the $\omega_{V_i}$ glue to
a global differential form $\omega_V\in\Omega^n_\rh(V)$ representing $\omega|_V$ in
$\Omega^n_\rh(V)\subset\Omega^n_\val(V)$.
%

Now let $Z\subset X$ be the exceptional locus of $V \to X$ and
let $E\subset V$ its preimage. By induction on the dimension there is
a class $\omega_Z$ in $\Omega^n_\rh(Z)$ mapping to $\omega|_Z\in\Omega^n_\val(Z)$. We know that $\omega_Z$ and $\omega_Y$ agree on $E$ because $\Omega^n_\rh(E) \to \Omega^n_\val(E)$ is injective and both represent
$\omega|_E$. By $\rh$-descent this gives a class
$\omega_X\in\Omega^n_\val(X)$ mapping to $\omega$.
\end{proof}

\begin{thm}\label{thm:compare}
The canonical morphisms of presheaves on $\Schft$
\[ \Omega^n_\rh\to\Omega^n_\cdh\to\Omega^n_\eh\to\Omega^n_\val\]
are isomorphisms. More generally, these are isomorphisms for any presheaf $\Fh$ as in Remark~\ref{rema:hYalso}. Moreover,
\[ \Omega^n(X)=\Omega^n_\val(X)\]
when $S = \Spec(k)$ for $X$ smooth.
\end{thm}

\begin{proof}
Corollary~\ref{cor:injective} says that the maps are injective. As the composition
is surjective, they are all isomorphisms. The smooth case is then a consequence of \cite[Theorem~5.11]{HKK}, which says that $\Omega(X) \cong \Omega_\rh^n(X)$ for smooth $X$. 
\end{proof}

\begin{rem}This still leaves open whether $\Omega^n_\eh\to\Omega^n_\dvr$
is an isomorphism. 
Weak resolution of singularities would imply that this is an isomorphism in general, see \cite[Proposition~5.13]{HKK}.
\end{rem}

\section{Quasi-coherence}\label{sec:quasi-coherence}
The sheaves $\Omega^n_\eh|_X$ are obviously sheaves of $\Oh_X$-modules.
We want to show that they are coherent. The main step is actually quasi-coherence.
\subsection{Quasi-coherence}
\begin{lemma}
Let $A$ be a finite type $k$-algebra, $f\in A$ not nilpotent, and $\omega\in\Omega^n_\val(A)$ be such that
$\omega|_{A_f}=0$. Then $f\omega=0$.
Moreover, the map 
\[ \Omega^n_\val(A)_f\to \Omega^n_\val(A_f)\]
is injective.
\end{lemma}

Note this would follow directly from quasi-coherence of $\Omega^n_\val|_{\Spec(A)}$. We want to prove it directly in order to show quasi-coherence down the line.

\begin{proof}
By torsion-freeness it suffices to show that $f \omega_{x} = 0$ for all residue field $\kappa$ of $\Spec(A)$, Lemma~\ref{lemm:rsF}. It suffices to consider the case $f|_\kappa \neq 0$. But then $A \to \kappa$ factors through $A_f$, and the claim follows from the assumption $\omega|_{A_f}=0$.

We now turn to injectivity. Let $\omega/f^N$ be in the kernel of
\[ \Omega^n_\val(A)_f\to \Omega^n_\val(\Spec A_f).\]
This means that $\omega_{R}/f^N=0$ for every $R \in \val(A_f)$.
As $f$ is invertible in this ring, this implies $\omega_R=0$.
That is, $\omega$ satisfies the assumption of the first assertion. 
Hence $f \omega=0$ in $\Omega^n_\val(A)$ and this implies $\omega/f^N=0$
in the localization.
\end{proof}

In order to proceed, we need a lemma from algebraic geometry.

\begin{lemma}\label{lem:envelope_exists}
Let $U\subset X$ be an open immersion of integral schemes, and let $V\to U$ be a $\cdp$-morphism with integral connected components. Then there is a cartesian diagram
\[ \begin{CD} V@>j>> Y\\
              @VVV @VVpV\\
              U@>>> X
\end{CD}\]
with $p$ a $\cdp$-morphism and $j$ an open immersion.
\end{lemma}
\begin{proof}
Let $V'$ be an irreducible component of $V$. We factor $V'\to U\to X$
as a (dense) open immersion $V'\to Y'$ followed by a proper map. It is easy
to see that $V'=Y'\times_XU$ because the map $V'\to U\times_XY'$
is both proper and a (dense) open immersion.
We define $Y$ as the disjoint union of all $Y'$ and $X\setminus U$.
\end{proof}

\begin{prop} \label{prop:integralQc}
Let $X$ be an integral $k$-scheme of finite type. Then the restriction to the small Zariski site $\Omega^n_\rh|_{X_\Zar}$ is quasi-coherent.\end{prop}
\begin{proof}
It suffices to show that for every integral ring $A$ and $f \in A$, the canonical morphism $\Omega^n_\rh(A)_f \to \Omega^n_\rh(A_f)$ is an isomorphism. By the last lemma and because $\Omega^n_\rh=\Omega^n_\val$, the map is injective. 

To show that it is surjective, it suffices to check that for every 
\[ \omega\in\Omega^n_\rh(A_f) \]
 there is $N$ such that $f^N\omega$ lifts to $\Omega^n_\rh(A)$. We put
$X=\Spec (A)$, $U=\Spec( A_f)$.
There is an $\rh$-cover $\{V_i\to U\}_{i=1}^m$ such that $\omega$ is represented
by an algebraic differential form on $\coprod V_i$. The strategy is to show that, up to multiplication by $f$, the form $\omega$ is actually representable on a $\cdp$-morphism of $U$, and then descend it to $U$. Lemma~\ref{lem:envelope_exists} is key.

We can choose the cover in the form
\begin{equation} \label{rhNormalForm}
 V_i\to V\to U
 \end{equation}
with $\pi:V\to U$ a $\cdp$-morphism, $V$ reduced and $V_i\to V$ open immersions. Moreover, we may assume that $V$ is disjoint union of its irreducible components and that they are birational over their image in $U$ (because we will want to apply Lemma~\ref{lem:envelope_exists}).

We will now construct the following cartesian squares whose vertical morphisms are proper envelopes, and horizontal ones are open immersions.
\[ \xymatrix{
\widetilde{W}_i \cap \widetilde{W}_j \ar[r] \ar[d] & \widetilde{W}_i \ar[d] \ar[r] & \widetilde{W} \ar[d] \\
\underset{\omega_i - \omega_j = 0 \  \in\  \Omega^n(W_{ij})}{W_{ij}} \ar[r] \ar[d] &\ar[r] \ar[d] & \widetilde{W}_{ij} \ar[d]^{p_{ij}} \\
V_i \cap V_j \ar[r] & \underset{\omega_i \ \in\  \Omega^n(V_i)}{V_i}  \ar[r] & V
} \]

Let $\omega_i\in\Omega^n(V_i)$
be the representing form. Since $\omega_i$ came from an element $\omega \in \Omega_\rh^n(U)$, the differences $\omega_i -\omega_j$ vanish in $\Omega_\rh^n(V_i \times_U V_j)$ by the exact sequence
\[ 0 \to \Omega_\rh^n(U) \to \bigoplus \Omega_\rh^n(V_i) \to \bigoplus \Omega_\rh^n(V_i \times_U V_j), \]
and hence vanish in $\Omega_\rh^n(V_i \times_V V_j) = \Omega_\rh^n(V_i \cap V_j)$. Consequently, there is
an $\rh$-cover of $V_i\cap V_j$ on which $\omega_i-\omega_j$ vanishes as a section of the presheaf $\Omega^n$. We can assume that this cover is again of the form \eqref{rhNormalForm}. Since $\Omega^n$ is Zariski separated, we find that there is a $\cdp$-morphism
$W_{ij}\to V_i\cap V_j$ that $\omega_i-\omega_j$ vanishes on $W_{ij}$.
By Lemma \ref{lem:envelope_exists} there is a commutative diagram
\[\begin{CD}
  W_{ij}@>j>> \widetilde{W}_{ij}\\
  @VVV     @VV{p_{ij}} V\\
 V_i\cap V_j@>>> V
\end{CD}
\]
with an open immersion $j$ and a proper envelope ${p_{ij}}$. 
Let $\widetilde{W} = \prod_V \widetilde{W}_{ij}$ be the fibre product of all $\widetilde{W}_{ij}$ over $V$. Hence $\widetilde{W}\to V$ is a proper envelope which factors through all $p_{ij}$.
Let $\widetilde{W}_i$ be the preimage of $V_i$ in $\widetilde{W}$. Now consider the above big diagram. The differences of the restrictions $\omega_i - \omega_j$ vanish in $\Omega^n(\widetilde{W}_i \cap \widetilde{W}_j)$, and $\{\widetilde{W}_i \to \widetilde{W}\}$, being the pullback of the Zariski cover $\{V_i \to V\}$ is also a Zariski cover. Hence, we can lift the restrictions $\widetilde{\omega}_i \in \Omega^n(\widetilde{W}_i)$ to a section 
\[ \omega_{\widetilde{W}} \in \Omega^n(\widetilde{W}). \]
We find that $\omega \in \Omega_{\rh}^n(U)$ and $\omega_{\widetilde{W}} \in \Omega^n(\widetilde{W})$ agree in $\Omega^n_\rh(\widetilde{W})$ since $\{\widetilde{W}_i \to \widetilde{W}\}$ is a Zariski cover, and they agree on $\widetilde{W}_i$. In other words, at this point, we have shown that $\omega|_{\widetilde{W}}$ is in the image of $\Omega^n \to \Omega^n_\rh$.

Again by Lemma \ref{lem:envelope_exists} there is a cartesian diagram
\[\begin{CD}
     \widetilde{W}@>j>> Y\\
    @VVV @VVp V\\
     U@>>> X
\end{CD}\]
with $j$ an open immersion and $p$ a proper envelope. 
The open subset $U$ in $X$ is by definition the
complement of $V(f)$, hence the same is true for $\widetilde{W}$ in $Y$. As
$\Omega^n_Y$ is coherent, this implies that there is $N$ such that
$f^N\omega_{\widetilde{W}}$ extends to $Y$. 

Let $\omega_Y$ be an extension of $f^N\omega_{\widetilde{W}}$ 
to $\Omega^n(Y)$.  Let $Z=V(f)\subset X$ and $E$ its preimage in $Y$. 
Consider the exact sequence
\[ 0\to\Omega^n_\rh(X)\to\Omega^n_\rh(Y)\oplus \Omega^n_\rh(Z)\to\Omega^n_\rh(E). \]
The class $(f\omega_Y,0)$ maps to $f\omega_Y|_E$. As $f=0$ in all of $E$, this
equals zero. Hence there is a class 
$\omega_{X}' \in \Omega_\rh^n(X)$ such that $\omega_X'|_{\widetilde{W}} = f^{N + 1}\omega_{\widetilde{W}}$ in $\Omega_\rh^n(\widetilde{W})$. Recall that two paragraphs ago we mentioned that $\omega|_{\widetilde{W}} = \omega_{\widetilde{W}}$ in $\Omega_\rh^n(\widetilde{W})$. So in fact, we know that $\omega_X'|_{\widetilde{W}} = f^{N + 1}\omega|_{\widetilde{W}}$ in $\Omega_\rh^n(\widetilde{W})$. But $\widetilde{W} \to U$ is a $\cdp$-morphism, so it follows that $\omega_X'|_{U} = f^{N + 1}\omega$ in $\Omega_\rh^n(U)$.
\end{proof}

\subsection{Coherence}

\begin{thm}\label{thm:coheren}
Let $X$ be of finite type over $k$. Then $\Omega^n_\rh|_{X_\Zar}$ is coherent.
\end{thm}

\begin{proof}
We will write $\Omega^n_\rh|_{X} = \Omega^n_\rh|_{X_\Zar}$ for briefness. 
Let $i:X_\red\to X$ be the reduction. We have $i_*(\Omega^n_\rh|_{X_\red}) =\Omega^n_\rh|_X$. Hence he may assume that $X$ is reduced. Let
$X=X_1\cup\dots\cup X_N$ be the decomposition into reduced irreducible components.
Let $i_j:X_j\to X$ and $i_{jl}:X_j\cap X_l\to X$ be the closed immersions.
By $\rh$-descent, we have an exact sequence of sheaves of $\OO_X$-modules
\[ 0\to \Omega^n_\rh|_X\to\bigoplus_{j=1}^Ni_{j*}(\Omega^n_\rh|_{X_j}) \to \bigoplus_{j,l=1}^Ni_{jl*}(\Omega^n_\rh|_{X_j\cap X_l})\]
By induction on the dimension it suffices to consider the irreducible case.

Let $X$ be integral with function field $K$. Let $\widetilde{X}$ be its normalization. The map $\pi:\widetilde{X}\to X$ is an isomorphism outside some
closed proper subset $Z\subset X$. Let $E$ be its preimage in $\widetilde{X}$.
From the blow-up sequence we obtain
\[ 0\to \Omega^n_\rh|_X\to \pi_*(\Omega^n_\rh|_{\widetilde{X}}) \oplus i_{Z*}(\Omega^n_\rh|_Z) \to \pi_*i_{E*}(\Omega^n_\rh|_E)\ .\]
Hence by induction on the dimension, we may assume that $X$ is normal.
We have shown in Proposition~\ref{prop:integralQc}, the sheaf $\Omega^n_\rh|_X$ is quasi-coherent in the integral case.
Let $j:X^{sm}\to X$ be the inclusion of the smooth locus with closed complement $Z$. It is of codimension
at least $2$. Hence $j_*\Omega^n_{X^{sm}}$ is coherent. 
By Theorem~\ref{thm:compare}, $\Omega^n_{X^{sm}} \cong \Omega^n_\rh|_{X^{sm}}$ so we have a map
$\Omega^n_\rh|_X \to j_*\Omega^n_{X^{sm}}$. Its image is coherent because it is a quasi-coherent subsheaf of a coherent sheaf. Its kernel $\cK$ is also quasi-coherent. We claim that it is a subsheaf of $i_{Z*}(\Omega^n_\rh|_Z)$. It suffices to prove that the canonical composition $\cK(U) \to \Omega^n_\rh|_X(U) \to i_{Z*}(\Omega^n_\rh|_Z)(U)$ is a monomorphism for all open $U \subseteq X$. Replacing $X$ with $U$, it suffices to consider the case $U = X$. Then this morphism is canonically identified with the morphism $\ker \biggl ( \Omega^n_\val(X) \to \Omega^n_\val(X^{sm}) \biggr ) \to \Omega^n_\val(Z)$, which is injective because $\Omega^n_\val(X) \to \Omega^n_\val(X^{sm}) \times \Omega^n_\val(Z)$ is injective, Lemma~\ref{lemm:rsF}. %
By induction on the dimension we can assume that
$i_{Z*}\Omega^n_\rh|_Z$ is also coherent, so the the kernel $\cK$ is coherent as well.
As a quasi-coherent extension of two coherent sheaves the sheaf $\Omega^n_\rh|_X$ is coherent.
\end{proof}

\subsection{Torsion}

We return to the question of torsion forms. 
As in \cite{HKK}, we denote by $\tor\ \Omega^n_\eh(X)$ the submodule of torsion sections, i.e., those vanishing on some dense open subset.
There is an obvious source of torsion classes: Let $f:Y\to X$ be proper birational with exceptional locus $Z\subset X$ and preimage $E\subset Y$. Any
$\omega\in\Ker(\Omega^n_\eh(Z)\to\Omega^n_\eh(E))$ gives rise to a torsion
class on $X$ by the blow-up sequence. By \cite[Example~5.15]{HKK} this kernel can indeed be nonzero. We have established in Lemma~\ref{lem:kill_torsion} that all torsion classes arise in this way.

\begin{prop}
Let $X$ be of finite type over $k$. Then:
\begin{enumerate}
\item The presheaf $\Th_X:U\mapsto\tor\ \Omega^n_\rh(U)$ is
a coherent sheaf of $\Oh_X$-modules on $X_\Zar$. 
\item There is a proper birational
morphism $f:Y\to X$ such that 
\[ \Th_Y=\Ker(\Omega^n_\rh|_{X_\Zar}\to f_*\Omega^n_\rh|_{Y_\Zar}).\]
\end{enumerate} 
\end{prop}
\begin{proof}
By the similar reductions as the first paragraph of the proof of Theorem~\ref{thm:coheren} it suffices to show coherence for $X$ integral, and since quasi-coherent subsheaves of coherent sheaves are coherent, it suffices to show quasi-coherence for $X$ integral.
 
Let $X=\Spec(A)$, $U=\Spec(A_f)$. We have to show that
$\Th_X(U)=\Th_X(X)_f$. Let $\omega\in\Th_X(U)\subset\Omega^n_\rh|_{X_\Zar}(U)=
(\Omega^n_\rh(X))_f$. Hence $\omega$ is of the form
$\tilde{\omega}/f^N$ with $\tilde{\omega} \in\Omega^n_\rh(X)$. By assumption, $\omega$ vanishes at the generic point of $U$, which is equal to the generic point of $X$. Hence  the same is true for $\tilde{\omega}$. This finishes the proof of coherence. 

Any form vanishing on a blow-up is torsion, i.e., $\Ker(\Omega^n_\rh|_{X_\Zar}{\to} f_*\Omega^n_\rh|_{Y_\Zar}) \subseteq \Th_Y$ for any proper birational $f:Y {\to} X$, and so our job is to find a $Y$ for which this inclusion is surjective. By Lemma~\ref{lem:kill_torsion}, we already know the existence of such a $Y$ for every single global torsion class. If $X$ is affine, the module $\Th_X(X)$ is finitely generated over
$\Oh(X)$. We can find a proper birational $Y \to U$ killing the generators and
hence all of $\Th_X(X)$, and by coherence even all of $\Th_X$. If $X$ is not affine, let $X=U_1\cup\dots \cup U_m$ be an affine cover, and $Y_i \to U_i$ proper birational morphisms killing all of $\Th_{U_i}$. By Nagata compactification, there is a proper birational morphism $\overline{Y_i} \to X$ such that $Y_i = U_i {\times_X} \overline{Y_i}$. Let $V$ be the closure of $\Spec(k(X))$ in $\overline{Y_1} {\times_X} \dots {\times_X} \overline{Y_m}$. It is equipped with the open cover $\{ U_i \times_X V \to V\}$. Moreover, each $U_i \times_X V \to V \to X$ factors through $Y_i$, so $\Th_X = \cup_{i = 1}^m \Ker(\Omega^n_\rh|_{X_\Zar}{\to} f_*\Omega^n_\rh|_{(U_i \times_X V)_\Zar})$. Then by Zariski descent we deduce that $\Th_X = \Ker(\Omega^n_\rh|_{X_\Zar}{\to} f_*\Omega^n_\rh|_{(V)_\Zar})$.
\end{proof}

It now becomes an interesting question to understand whether a given
$X$ admits such a blow-up $Y\to X$ such that there is a point $\xi$ in the exceptional locus over which all residue fields of $Y_\xi$ are inseparable over $\kappa(\xi)$. This does not
happen in the smooth case: \emph{any} blowup of a regular scheme is completely decomposed (unconditionally), \cite[Prop.2.12]{HKK}. In the example in \cite[Example 3.6]{HKK} the point
$\xi$ had codimension $1$ and $Y$ was the normalisation. it induced a purely inseparable field extension of $k(\xi)$.

One might wonder if assuming $X$ is normal is enough to avoid this pathology. Let us show that it is not.

\begin{prop} \label{prop:counterexample}
There exists a normal variety $X$ over a perfect field $k$ of positive characteristic $p$, a point $\xi \in X$  and a blow-up $Y\to X$ such that
for every point in the fibre $Y_\xi\to \xi$ the residue field extension
is inseparable.

In particular, for every $X$-valuation ring $R$ of $k(X)$ sending its special point to $\xi$, the field extension 
$k(\xi) \subseteq R / m$ is inseparable.
\end{prop}

\begin{proof}
Our variety is
\[ X = \Spec \left ( \frac{k[s, t, x, y, z]}{z^p - s x^p - t y^p} \right ), \]
from \cite[Example 3.5.10]{SV}. Let $Y$ be the blowup of this variety at the
 ideal $(x, y, z)$. The blowup $Y$ admits an open affine covering by affine schemes with rings
\[ 
\frac{k[s, t, \tfrac{x}{z}, \tfrac{y}{z}, z]}{1 - s (\tfrac{x}{z})^p - t (\tfrac{y}{z})^p}, \qquad  
\frac{k[s, t, \tfrac{x}{y}, y, \tfrac{z}{y}]}{(\tfrac{z}{y})^p - s (\tfrac{x}{y})^p - t}, \qquad  
\frac{k[s, t, x, \tfrac{y}{x}, \tfrac{z}{x}]}{(\tfrac{z}{x})^p - s - t (\tfrac{y}{x})^p},
\]
and the intersections of these open affines with the exceptional fibre $E$ are
\[ 
\frac{k[s, t, \tfrac{x}{z}, \tfrac{y}{z}]}{1 - s (\tfrac{x}{z})^p - t (\tfrac{y}{z})^p}, \qquad  
\frac{k[s, t, \tfrac{x}{y}, \tfrac{z}{y}]}{(\tfrac{z}{y})^p - s (\tfrac{x}{y})^p - t}, \qquad  
\frac{k[s, t, \tfrac{y}{x}, \tfrac{z}{x}]}{(\tfrac{z}{x})^p - s - t (\tfrac{y}{x})^p},
\]
respectively, all lying over the singular locus 
$V(x, y, z) = \Spec(k[s, t]) \subseteq X$. Our point $\xi$ is the generic point $\xi = \Spec(k(s, t))$. Every point of the fibre $E_\xi$ has residue field an inseparable extension of $k(\xi)$: Consider for example the fibre of the right most affine for concreteness (all three fibres are isomorphic up to an automorphism of $k(s, t)$). It factors as
\[ \Spec\left ( \frac{k(s, t)[\tfrac{y}{x}, \tfrac{z}{x}]}{(\tfrac{z}{x})^p - s - t (\tfrac{y}{x})^p} \right ) \to \Spec(k(s, t)[\tfrac{y}{x}]) \to \Spec(k(s, t)). \]
Any residue field $k(\zeta)$ of a point $\zeta$ of the left most affine scheme is a finite field extension of the residue field $k(\theta)$ of a point $\theta \in \AA^1_{k(s, t)}$ in the middle. This latter $k(\theta)$ is generated over $k(s, t)$ by the image of $\tfrac{y}{x}$. This finite field extension $k(\zeta) / k(\theta)$ is purely inseparable so long as $s+t(\tfrac{y}{x})^p$ is nonzero in $k(\theta)$. If $k(\theta) = k(s, t, \tfrac{y}{x})$, then clearly this is the case. If $k(\theta) = k(s, t)[\tfrac{y}{x}] / f(\tfrac{y}{x})$ for some irreducible polynomial $f(\tfrac{y}{x}) \in k(s, t)[\tfrac{y}{x}]$, then $s+t(\tfrac{y}{x})^p$ is nonzero if and only if $(\tfrac{y}{x})^p = -\tfrac{s}{t}$ mod $f(\tfrac{y}{x})$. But since $(\tfrac{y}{x})^p  {+} \tfrac{s}{t}$ is irreducible in $k(s, t)[\tfrac{y}{x}]$, this would imply $f(\tfrac{y}{x}) = (\tfrac{y}{x})^p  {+} \tfrac{s}{t}$, in which case the subextension $k(\zeta) / k(\theta) / k(s, t)$ is already purely inseparable.

Now the blowup $Y \to X$ is birational and proper, and so any $X$-valuation ring $R$ of $k(X)$ is uniquely a $Y$-valuation ring of $k(Y)$, and if the special point of $R$ is sent to $\xi$, then the lift sends this special point to some point $\eta \in E_\xi$. But any field extension which contains an inseparable field extension is inseparable, and $k(\xi) \to R / \m$, contains $k(\xi) \to k(\eta)$, hence, $k(\xi) \to R / \m$ is inseparable.
\end{proof}

\begin{rema}
Let us also observe that the rightmost open affine contains the point $\xi' = V(\tfrac{y}{x})$ of $E_\xi$ with residue field $k(s, t)[\tfrac{z}{x}] / ((\tfrac{z}{x})^p {-} s) \cong k(s^{1/p}, t)$. In particular, we have produced a blowup $Y \to X$, a point $\xi \in X$, and a point $\xi' \in Y$ over it for which $\Omega^1(k(\xi)) \to \Omega^1(k(\xi'))$ is neither injective, nor zero.
\end{rema}
\subsection{The \'etale case}
We recall:
\begin{defn}Let $X\in\Schft$. A presheaf $\Fh$ of $\Oh$-modules on the small \'etale site $X_\et$ is
called \emph{coherent} if for all $X'$ \'etale over $X$, the sheaf $\Fh|_{X'_\Zar}$ is a coherent sheaf for the Zariski-topology
and, in addition, for all $\pi:X_1\to X_2$ \'etale, the natural map
\[ \left(\pi^*\Fh|_{X_{2,\Zar}}\right)\tensor_{\pi^*\Oh_{X_2}}\Oh_{X_1}\to \Fh|_{X_{1,\Zar}}\]
is an  isomorphism.
\end{defn}
The left hand side is nothing but the pull-back in the category of quasi-coherent sheaves. We reserve the notation $\pi^*$ for the pull-back of abelian sheaves.

\begin{prop} \label{prop:etcoh}
For all $X\in\Schft$, the sheaf $\Omega^n_\eh|_{X_\et}$ is
coherent. 
\end{prop}

\begin{proof}
We already know coherence for the Zariski-topology. By replacing
$X_2$ by $X$, it suffices to check the condition for all \'etale
morphisms $X'\to X$. Both sides are sheaves for the Zariski-topology, hence
if suffices to consider the case where $X$ and $X'$ are both affine
and $\pi: X'\to X$ is surjective. (Indeed, first assume $X$ and $X'$ affine, then cover the image of $X'$ by affines $U_i$ and replace $X$ by $U_i$ and
$X'$ by the preimage of $U_i$.)
We fix such a $\pi$.

We need to show that
\begin{equation}\label{eq:todo} \Omega^n_\eh(X)\tensor_{\Oh(X)}\Oh(X')\to  \Omega^n_\eh(X')\end{equation}
is an isomorphism.
Note that the $\Omega^n$ analogue of this assertion holds.

{\it Step 1:} Consider the morphism of presheaves on the category of separated schemes of finite type over $X$
\begin{equation}\label{eq:quasi}Y\mapsto  \left(\Omega^n(Y)\tensor_{\Oh(X)}\Oh(X')\to \Omega^n(Y\times_XX')\right).\end{equation}
We claim that it is an isomorphism. Indeed: Both sides are sheaves
for the Zariski-topology (the left hand side because $\Oh(X)\to\Oh(X')$ is
flat). Hence it suffices  to consider the
case where $Y$ is affine. Let $Y'=Y\times_XX'$. It is also affine. Then the
map (\ref{eq:quasi}) identifies as
\[ \Omega^n(Y)\tensor_{\Oh(X)}\Oh(X')=\Omega^n(Y)\tensor_{\Oh(Y)}\Oh(Y')\to \Omega^n(Y'),\]
hence it is an isomorphism.

{\it Step 2:} We now sheafify the morphism of presheaves with respect to the $\eh$-topology. As
$\Oh(X)\to\Oh(X')$ is
flat, it commutes with sheafification, and  we get an isomorphism of $\eh$-sheaves
\begin{equation}\label{eq:ehstep2}
 \Omega^n_\eh\tensor_{\Oh(X)}\Oh(X')\to (\Omega^n(\cdot\times_XX'))_\eh.
\end{equation}

{\it Step 3:}
We claim that
\begin{equation}\label{eq:ehstep3} (\Omega^n(\cdot\times_X X'))_\eh(Y)=\Omega^n_\eh(Y\times_XX').\end{equation}
It suffices to show that every $\eh$-cover of
$Y'=Y\times_XX'$ can be refined by the pull-back of an $\eh$-cover of $Y$.
Let $Z\to Y'$ be an $\eh$-cover. The composition $Z\to Y'\to Y$ is
also an $\eh$-cover. The pull-back $Z\times_YZ\to Z\to Y'$ is
the required refinement.

Combining the isomorphisms \eqref{eq:ehstep2} and \eqref{eq:ehstep3} in the case $Y = X$ gives the desired isomorphism \eqref{eq:todo}.
\end{proof}

\section{Special cases} \label{sec:specialCases}
The cases of forms of degree zero or top degree are easier to handle than
the general case. In this section we study these special cases.

\subsection{$0$-differentials}

This section is about the sheafification of the presheaf $\Omega^0 = \OO$. For expositional reasons we work over a field. For more general bases, and more general representable presheaves, see Section~\ref{ssec:representable}.
 
In contrast to the general case, $\OO_\val$ is torsion-free.

\begin{lemma} \label{lemm:OvalTorsionFree}
For every $X \in \Schft$, every dense open immersion $U \subseteq X$ in $\Schft$ induces an inclusion 
$\Oh_\val(X) \subseteq \Oh_\val(U)$.
\end{lemma}

\begin{proof}
First suppose it is true for irreducible schemes, and let
$X_1,\dots,X_n$ be the irreducible components of $X$. Let
$U_1,\dots,U_n$ be their intersections with $U$. Each $U_i$ is dense in
$X_i$. Let $ f, g \in \OO_\val(X)$ such that $f|_U= g|_U$. Then $f|_{U_i}=g|_{U_i}$. 
By the irreducible case, $f|_{X_i}=g|_{X_i}$. By $\cdh$-descent
\[ \OO_\val(X) \subseteq \OO_\val(X_i)\]
and hence $f=g$.

Now consider $X$ irreducible. Since $\OO_\val(X) = \OO_\val(X_\red)$ we can assume $X$ is integral. Consider the description of Lemma~\ref{lemm:rsF}.
 If two sections $(s_x), (t_x) \in \OO_\val(X)$ are equal on a dense open, then $s_\eta = t_\eta$ where $\eta$ is the generic point of $X$.
 Consequently, for any valuation ring $R$ of the form $\eta \to \Spec(R) \to X$, the lifts $s_{\Spec(R)}, t_{\Spec(R)}$ also equal, as well as their images in 
$R / m$, and from there we deduce that $s_x = t_x$ where $x = $ image $\Spec(R/\m) \to X$.  
But for every point $x\in X$ there is a valuation
ring $R_x\subset k(X)$ such that the special point of $\Spec(R)$ maps
to $x$, \cite[7.1.7]{EGAII}.
\end{proof}

Recall the notion of the seminormalisation of a variety, see Definition~\ref{defn:sn}.

\begin{prop} \label{prop:ehSN}
Let $Y \in \Schft$. 
Then the canonical morphism
\[ \OO_{\val}(Y) 
\cong
\OO(Y^\sn)
\]
is an isomorphism. The same is true for $\OO_\rh$, $\OO_\cdh$, and $\OO_\eh$ in place of $\OO_{\val}$.
\end{prop}

\begin{rema} \label{rema:ehSN}
The below proof works for a general noetherian base scheme $S$. We (the authors) do not know if the seminormalisation $A^{\sn}$ is always noetherian in this setting, but the definition of $\OO_{\val}(A^{\sn})$ is, clearly, still valid, and the careful reader will not that the proof below still works, regardless.
\end{rema}

\begin{proof}
Both $\OO_\val(-)$ and $\OO((-)^{\sn})$ are invariant under $(-)_\red$, and are Zariski sheaves, so it suffices to consider the case where $Y$ is reduced and affine. Let $Y = \Spec(A)$. As $\Spec(A^{\sn}) {\to} \Spec(A)$ is a completely decomposed homeomorphism, $\val(A^{\sn}) \to \val(A)$ is an equivalence of categories, so $\OO_{\val}(A) \to \OO_{\val}(A^{\sn})$ is an isomorphism. Hence, it suffices to show that if $A$ is a seminormal ring, then $\OO(A) = \OO_{\val}(A)$. Define $A_\val = \OO_\val(\Spec(A))$. By Lemma~\ref{lemm:seminormalProperties}\eqref{lemm:seminormalProperties:subInt} it suffices to show that $A \to A_\val$ is subintegral.

First we show that it is integral. By the argument in Lemma~\ref{lemm:OvalTorsionFree}, or by its statement and the comparison of Theorem~\ref{thm:compare}, both of the canonical morphisms $A \to A_\val \to \prod_{\p \textrm{ minimal}} A_\p$ is an embedding into a product of fields, \cite[Tag 00EW]{stacks-project}.
Since $\Spec(A)$ is a noetherian topological space, there are finitely many of them. 
Now by the definition of $A_\val$, the image of $A_\val$ in each $A_\p$ is contained in any valuation ring of $A_\p$ containing the image of $A$. Since the normalisation is the intersection of these valuation rings, \cite[Tag 090P]{stacks-project}, it follows that the extension $A \subseteq A_\val$ is integral. 

To show that $\Spec(A_\val) \to \Spec(A)$ is a completely decomposed homeomorphism, it suffices to show that for all fields $\kappa$ which are residue fields of $A$ or $A_\val$, 
\begin{equation} \label{equa:AvalAhom}
\hom_k(A_\val,  \kappa) \to \hom_k(A, \kappa)
\end{equation}
is an isomorphism. 

Surjectivity. We construct a section of the map of sets \eqref{equa:AvalAhom}. For every morphism $\phi: A \to \kappa$ in $\val(A)$ to a residue field $\kappa$ of $A$, there is a canonical extension $\phi: A \stackrel{\iota}{\to} A_\val \stackrel{\pi_{\phi}}{\to} \kappa$ making the triangle commute: just take $\pi_\phi: \varprojlim_{\phi':A \to R' \in \val(A)} R' \to \kappa$ to be the projection to the $\phi$th component of the limit. For $\kappa$ a general field, take $\pi_\phi$ to be the map associated to the residue field corresponding to $\phi$. 

Injectivity. We show that the section $\phi \mapsto \pi_\phi$ we have just constructed is surjective. That is, for an arbitrary residue field $\psi: A_\val \to \kappa$, we claim that $\psi = \pi_{A{\to}A_\val{\to}\kappa}$. If $\psi \circ \iota$ is the canonical map $A \to A_\p$ to one of the fractions fields of an irreducible component of $\Spec(A)$, then there is a unique lift because $A \subseteq A_\val \subseteq  \prod_{\p \textrm{ minimal}} A_\p$, as we observed above. For a general residue field of $A_\val$, there is an $A_\val$-valuation ring $R \subseteq A_\p$ of the fraction field $A_\p$ of any irreducible component containing $\Spec(\kappa) \in \Spec(A_\val)$ whose special point maps to $\Spec(\kappa)$, \cite[7.1.7]{EGAII}, (for the non-noetherian version, cf. \cite[Tag 00IA]{stacks-project}). So we have the commutative diagram
\[ \xymatrix{
R / \m & \ar[l] R \ar[r] & A_\p \\
\kappa \ar[u] & \ar[l]_{\psi}  A_\val \ar[u] \\
} \]
As we have just discussed, the map $A_\val \to A_\p$ must be the unique extension $\pi_{A {\to} A_\val {\to} A_\p}$ of the canonical $A \to A_\p$. As $R \to A_\p$ is injective, the map $A_\val \to R$ must also be the canonical $\pi_{A{\to}A_{\val}{\to}R}$, and therefore $A_\val \to R / \m$ is $\pi_{A{\to}A_{\val}{\to}R{\to}R/\m}$, and by injectivity of $\kappa \to R / \m$, we conclude $\psi = \pi_{A{\to}A_\val{\to}\kappa}$,

The claim about $\OO_{\rh}, \OO_{\cdh}, \OO_{\eh}$ follows from Theorem~\ref{thm:compare}.
\end{proof}

Recall the $\sdh$-topology introduced in \cite[Section~6.2]{HKK}.
It is generated by \'etale covers and those proper surjective maps that
are separably decomposed, i.e., any point has a preimage such that the residue field extension is finite and separable. By de Jong's theorem on alterations, see \cite{deJ96}, every $X\in\Schft$ is $\sdh$-locally smooth. However, $\sdh$-descent fails for differential forms, \cite[Proposition 6.6]{HKK}. The situation is better in degree zero.

\begin{prop}\label{prop:sdh_val}
Let $k$ be perfect. Then
$\Oh_\val = \Oh_\sdh$.
\end{prop}

\begin{proof}
The $\sdh$-topology is stronger than the $\eh$-topology, and we know $\OO_\val {=} \OO_\eh$, Theorem~\ref{thm:compare}. Hence, we have a canonical morphism 
$\OO_\val \to \Oh_\sdh$, and an isomorphism $(\OO_\val)_\sdh \cong \OO_\sdh$. If we can show that $\OO_\val$ is already an $\sdh$ sheaf, we are done.
The topology is generated by proper separably decomposed morphisms and \'etale covers. We already know that $\OO_\val$ is an \'etale sheaf. Hence
it suffices to show: If $Y \to X$ is a proper $\sdh$-cover, which generically is finite and separable, then
\[ 0\to \Oh_\val(X)\to\Oh_\val(Y)\to\Oh_\val(Y\times_X Y)\]
is exact.

Recall that since $\OO$ is torsion-free on valuation rings, $\Oh_\val(X) \to \prod_{x \in X} \kappa(x)$ is injective, Lemma~\ref{lemm:rsF}. For $y\in Y$ with image $x\in X$, the induced map $\kappa(x)\to\kappa(y)$ is injective as a map of fields. As $Y\to X$ is surjective, this implies that 
$\Oh_\val(X)\to\Oh_\val(Y)$ is injective.

Let $f\in\Oh_\val(Y)$ be in the kernel of the second map.
We want to define an element $g\in\Oh_\val(X)$ and start with the 
component
$g_x\in \kappa(x)$ for $x\in X$. Let $y\in Y$ be a preimage of $X$ such
that the residue field extension $\kappa(y) / \kappa(x)$ is finite and 
separable. Let $\lambda / \kappa(x)$ be a finite Galois extension containing $\kappa(y)$. We have a canonical map $l: \Spec(\lambda) \to Y$, and any 
$\kappa(x)$-automorphism $\sigma$ of $\lambda$ gives us a second map $l \circ \sigma$, and this pair of maps define some $(l, l \circ \sigma): \Spec(\lambda) \to Y\times_XY$. By the assumption that $f$ is a cocycle, $f_l = (\pi_1^*f)_{(l,l\sigma)} = (\pi_2^*f)_{(l,l\sigma)} = f_{l\sigma}$ in $\lambda$. That is,
 $f_l \in \lambda$ is $\Gal(\lambda / \kappa(x))$-invariant, and therefore, 
actually lies in $\kappa(x) \subseteq \lambda$. %
We define $g_{\kappa(x)}:=f_{\lambda}$.

Note that another consequence of the cocycle condition is that $g_{\kappa(x)}$ is independent of the choice of $y$, even without assuming separability or 
finite. For any other $y'$ over $x$, we can choose an extension $K / \kappa(x)$ containing both $\kappa(y)$ and $\kappa(y')$, leading to a map 
$\Spec(K) \to Y \times_X Y$, to which we can apply the cocycle condition to 
find that $f_{y} = f_{y'}$ in $K$ via the chosen embeddings, and therefore 
they also agree in $\kappa(x)$.

It remains to show that the tupple $(g_{\kappa(x)})_{x\in X}$ defines a
section of $\Oh_\val(X)$. We continue with the criterion of Lemma~\ref{lemm:rsF}. 
Let $R\subset \kappa(x)$ be a valuation ring over $X$.
Let $y\in Y$ again be a preimage of $x$. There is a valuation ring $S\subset\kappa(y)$ such that $R=S\cap \kappa(x)$, \cite[Ch.VI §3, n.3, Prop.5]{Bou64}. By the valuative criterion for properness, $S$ is a $Y$-valuation on $\kappa(y)$. As $f \in \Oh_\val(Y)$, the element $g_{\kappa(x)} = f_{\kappa(y)}$ is in $S \subseteq \kappa(y)$, but it is also in $\kappa(x)$, so $g_{\kappa(x)}$ is in $R \subseteq \kappa(x)$. Therefore, let us write it as $g_R$. Let $x_0$ (resp. $y_0$) be the image of $\Spec(R/\m_R) \to X$ (resp. $\Spec(S/\m_S) \to Y)$. To finish, we must show that $g_{R}$ agrees with $g_{x_0}$ in $R / \m$. But 
$g_{R}|_{R/\m_R}|_{S/\m_S} = g_{R}|_S|_{S/\m_S} = f_S|_{S/\m_S} = f_{S/\m_S} = f_{y_0}|_{S/\m_S} = g_{x_0}|_{y_0}|_{S/\m_S} = g_{x_0}|_{R/\m_R}|_{S/\m_S}$ 
and $R/\m_R \to S/\m_S$ is injective, so $g_R|_{R/\m_R} = g_{x_0}|_{R/\m_R}$.
\end{proof}

\begin{rema}
Let us point out where the above proof breaks for $\Omega^n$. The argument for injectivity is actually valid because we can choose the preimage $y$ of $x$ to be separable. The construction of each $g_{\kappa(x)}$ is fine, as well as independence of the choice of $y$ 
used in the construction. However, for $y$ over $x$ which are not separable, we cannot necessarily check that $g_{x}|_{y} = f_{y}$. Choosing $y / x$ 
separable in the last paragraph, we can show that each $g_{\kappa(x)}$ lifts to any $X$-valuation ring of $\kappa(x)$, but we cannot ensure that 
$R/\m_R \to S/\m_S$ is separable, nor its image $x_0 \to y_0$, so we cannot check that we have a well-defined section.

In fact, not being able to control this kind of ramification is precisely why the $\sdh$-topology is not suitable for working with differential forms, cf. \cite[Example 6.5]{HKK}.

On the other hand, Proposition~\ref{prop:sdh_val} is valid for any representable presheaf $h^Y$ for any scheme $Y$. Moreover, using the same proof, we can show that $\Omega^n_\val(X) {=} \Omega^n_\sdh(X)$ whenever $\dim X \leq n$. 
\end{rema}

\begin{prop} \label{prop:sdh_dvr}
$\Oh_\sdh=\Oh_\dvr$.
\end{prop}
\begin{proof}The same arguments as in the last proof show that
$\Oh_\dvr$ has $\sdh$-descent. (In the above notation: if $
R$ is a discrete valuation ring, then $S$ can also be chosen as a discrete
valuation ring).
As pointed out before, any $X\in\Schft$ is smooth locally for the
$\sdh$-topology. Hence it suffices to compare the values on smooth
varieties. In this case we have on the one hand $\Oh_\sdh(X)=\Oh_\val(X)=\Oh(X)$,
on the other hand $\Oh_\dvr(X)=\Oh(X)$ by see \cite[Remark~4.3.3]{HKK}. 
\end{proof}

\begin{rem}\label{rem:all_n=0}Hence we have
\[ \Oh_\rh = \Oh_\cdh = \Oh_\eh = \Oh_\val = \Oh_\dvr = \Oh_\sdh.\]
However, in positive characteristic $\Oh_\h\neq \Oh_\val$ because $\Oh_\val$ does not have descent for Frobenius covers. Cf. Proposition~\ref{prop:ehSN}, Remark~\ref{rem:OhIsPerfectClosure}.
\end{rem}

The following property is well-known for the ordinary structure sheaf under the assumption that $Y$ is normal. It will be useful in connection with cohomological descent questions, cf. \cite{HK}.

\begin{prop}\label{prop:conn_iso}Let $Y'\to Y$ be a $\cdp$-morphism in $\Schft$ with geometrically connected 
fibres. Then
\[ \Oh_\eh(Y)\to \Oh_\eh(Y')\]
is an isomorphism.
\end{prop}

\begin{proof}
We use induction on the dimension of $Y$. Note the hypotheses are preserved by all base-changes along $Y$. Without loss of generality, all schemes are assumed to be reduced. If $\dim Y = -1$ (i.e., $Y = \varnothing$), then $Y' = \varnothing$ and the proposition follows from $\OO_\eh(\varnothing) = 0$. In dimension $\geq 0$, let $\pi:\widetilde{Y}$ be the normalisation of $Y$. Let $Z\subset Y$ be the locus  where
$\pi$ fails to be an isomorphism and $E=\pi^{-1}Z$ the preimage. 
Let $\widetilde{Y}'$, $Z'$ and $E'$ be the basechanges to $Y'$. We have a commutative
diagram of blow-up sequences
\[\begin{xy}\xymatrix{
0\ar[r]&\Oh_\eh(Y')\ar[r]&\Oh_\eh(\widetilde{Y}')\oplus \Oh_\eh(Z')\ar[r]&\Oh_\eh(E')\\
0\ar[r]&\Oh_\eh(Y)\ar[r]\ar[u]&\Oh_\eh(\widetilde{Y})\oplus \Oh_\eh(Z)\ar[r]\ar[u]&\Oh_\eh(E)\ar[u]
}\end{xy}\]
By the induction hypothesis, it now suffices to prove $\Oh_\eh(\widetilde{Y}) \cong \Oh_\eh(\widetilde{Y}')$. 
By Proposition~\ref{prop:ehSN} $\OO_\eh((-)^\sn) = \OO((-)^\sn)$, 
so it suffices, in fact, to show that $\OO(\widetilde{Y}'^{\sn}) \to \OO(\widetilde{Y})$ is an isomorphism. 

Since $\widetilde{Y}'^{\sn} \to \widetilde{Y}'$ is a proper homeomorphism, we are dealing with a proper surjective morphism $\widetilde{Y}'^{\sn} \to \widetilde{Y}$ to a normal scheme. In this situation, Stein factorisation \cite[Prop.4.3.1]{EGAIII1} gives us a factorisation $\widetilde{Y}'^{\sn} \to W \to \widetilde{Y}$ such that $\OO(\widetilde{Y}'^{\sn}) \cong \OO(W)$ and such that $W \to \widetilde{Y}$ is finite. 
Since $\widetilde{Y}'^{\sn} \to \widetilde{Y}$ has connected fibres, so does $W \to \widetilde{Y}$ \cite[Cor.4.3.3]{EGAIII1}. We also deduce that because $\widetilde{Y}'^{\sn}$ is reduced so is $W$, and because $\widetilde{Y}'$ is completely decomposed, so is $\widetilde{Y}'^{\sn} \to \widetilde{Y}$ and
 therefore $W \to \widetilde{Y}$ also. In particular, since $W \to \widetilde{Y}$ is completely decomposed and $W$ reduced, the fibre over the generic point of $\widetilde{Y}$ must be an isomorphism. Replacing $W$ with its normalisation, $\widetilde{W} \to \widetilde{Y}$, we have a finite birational morphism between normal schemes. This can only be an isomorphism, so $W \to \widetilde{Y}$ was an isomorphism, and $\OO(W) = \OO(\widetilde{Y})$. 

To summarise, $\OO_\eh(\widetilde{Y}') \cong \OO_\eh(\widetilde{Y}'^{\sn}) \cong \OO(\widetilde{Y}'^{\sn}) \cong \OO(W) \cong \OO(\widetilde{Y})$.
\end{proof}

\subsection{Top degree differentials}
Recall the notion of a birational morphism of schemes in the non-reduced case from Section~\ref{ssec:not}

\begin{prop} \label{prop:top} 
Let $X/k \in \Schft$ be of dimension at most $d$. 
\begin{enumerate}
\item $\Omega^d_\cdh(X) $ is a birational invariant, i.e., it remains unchanged under proper surjective birational morphisms. 
\item \label{prop:top:3} We have
\[ \Omega^d_\cdh(X)=\varinjlim_{X' \to X}\Omega^d(X')\]
where the colimit is over proper surjective birational morphisms $X' \to X$.
\item
Elements of $\Omega^d_\cdh(X)=\Omega^d_\val(X)$ are determined by their value on the total ring of fractions $Q(X)$, and the integrality condition only needs to be tested
on valuations of the function fields. In particular, it is torsion-free. 
\item More precisely, if $X$ is irreducible of dimension $d$, then, cf. Lemma~\ref{lemm:rsF},
\begin{equation} \label{equa:Omegaddimdintegral} 
\Omega^d_\val(X) = \bigcap_{\Spec(k(X)) {\to} \Spec(R) {\to}X \in \rval(X)} \Omega^d(R) 
\end{equation}
In general, if $X=X_1\cup \dots\cup X_N$ is the decomposition into irreducible
component, then
\[ \Omega^d_\val(X)=\bigoplus_{i=1}^N\Omega^d_\val(X_i).\]
\[\Omega^d_\val(X)=\Omega^d_\val(\widetilde{X}).\]

\end{enumerate}
\end{prop}
\begin{proof} 
Note that $\Omega^d_\cdh=\Omega^d_\val$ vanishes on schemes of dimension less that $d$. Hence the first statement is immediate from the sequence
for abstract blow-up squares.

The third statement follows from Lemma~\ref{lemm:rsF}, the fact that $\Omega^n(K) = 0$ for $n >\mathrm{trdeg}(K/k)$, and the valuative criterion for properness.
The explicit formula is immediate from this.

For the second statement consider
\[ X \mapsto \widetilde{\Omega}^d(X):=\varinjlim_{X' \to X}\Omega^d(X').\]
By definition this is a birational invariant. 
We claim that $\widetilde{\Omega}^d$ is torsion-free. Note that $X'$ can always be refined by the disjoint union of its irreducible components with their reduced structure. Let $\omega$ be a torsion element of $\widetilde{\Omega}^d(X)$. It is represented by a differential form on some $X_1\to X$. After restriction to some further $X_2\to X_1$ it vanishes on a dense open subset. Then there is a proper birational morphism $X_3\to X_2$ such
that $\omega|_{X_3}=0$. This was shown in  \cite[Theorem~A.3]{HKK}
(for a recap see Theorem~\ref{theo:A3} combined with Theorem~\ref{thm:hypV}).
Hence $\omega=0$ in the direct limit. 

By torsion-freeness, we have $\widetilde{\Omega}^d(X)=0$ if the dimension of
$X$ is less than $d$. Hence $\widetilde{\Omega}^d$ is a presheaf on the category
of $k$-schemes of dimension at most $d$. It is Zariski-sheaf because
$\Omega^d$ is. It has descent for abstract blow-up squares by birational invariance and vanishing in smaller dimensions. Hence it is an $\rh$-sheaf. By the
universal property, there is a natural map
\[ \Omega^d_\rh \to \widetilde{\Omega}^d.\]
The map 
\[ \widetilde{\Omega}^d(X) = \varinjlim_{X'\to X}\Omega^d(X') \to \varinjlim_{X' \to X}\Omega^d_{\rh}(X') = \Omega^d_{\rh}(X) \]
induces a natural map in the other direction. We check that they are inverse to each other. Both sheaves are torsion-free, hence it suffices to consider
generic points where it is true. 
\end{proof}

\begin{rema}
Note that the description in Equation~(\ref{equa:Omegaddimdintegral}) can also be interpreted as the global sections on the Riemann-Zariski space $\Gamma(\RZ(X), \Omega^d_{RZ(X)/k})$.
\end{rema}

In the smooth case, this gives a formula involving only ordinary differential forms.

\begin{coro} \label{coro:mainDim}
Let $X$ be a smooth $k$-scheme of dimension $d$. Then 
\[ \Omega^d(X) = \varinjlim_{X' \to X} \Omega^d(X') \]
where the colimit is over proper birational morphisms $X' \to X$.
\end{coro}

\begin{prop}\label{prop:sdh_top}
On the category of $k$-schemes of dimension at most $d$, we have
\[ \Omega^d_\val = \Omega^d_\eh=\Omega^d_\sdh=\Omega^d_\dvr.\]
\end{prop}

\begin{proof}	
The first isomorphism is Theorem~\ref{thm:compare}. For the other two, the same proofs as for Proposition~\ref{prop:sdh_val} and Proposition~\ref{prop:sdh_dvr} work.
\end{proof}

\subsection{Representable sheaves}\label{ssec:representable}

Note that $\Omega^0 = \OO = \hom(-, \AA^1)$. In this section we extend our results
to all representable sheaves over a general noetherian base $S$. 
We will use the following notation for representable presheaves on
$\Schft[S]$.
\[ h^Y(-) = \hom_{\Sch[S]}(-, Y), \qquad Y \in \Schft[S]. \]
Note that this presheaf satisfies the properties of Remark~\ref{rema:hYalso}.
Notice also that the $h^Y$ are torsion free in the sense of Definition~\ref{defn:torsion-free}---this is exactly the valuative criterion for separatedness.

\begin{lemm} \label{lemm:cdpushout}
Suppose that the noetherian base scheme $S$ is Nagata. %
Let $X' {\to} X$ be a finite completely decomposed surjective morphism in $\Sch[S]$, and suppose that $X'$ is seminormal. Then the coequalisers
\[ 
\coeq(X' {\times_X} X' \rightrightarrows X') = C, \qquad
\coeq((X' {\times_X} X')^\sn \rightrightarrows X') = D, \]
exist in $\Sch[S]$, we have $D = X^{\sn}$ and the canonical morphisms $D \to C \to X$ are finite completely decomposed \emph{homeo}morphisms.
\end{lemm}

\begin{proof}
Using the description \cite[Sco.4.3]{Fer03}, one easily constructs the coequaliser in the category of locally ringed spaces by taking the coequaliser in the category of sets, equipping it with the quotient topology, and the equaliser of the direct images of the structure sheaves. Using this description, one readily deduces from $X' \to X$ being finite that $D \to C \to X$ are homeomorphisms. 
Note that $X^{\sn} \to X$ is also a homeomorphism. Now since $X' \to X$ is an $\rh$-cover, it follows from Proposition~\ref{prop:ehSN}, Remark~\ref{rema:ehSN}, that $\OO(X^{\sn}) = \eq(\OO(X') \rightrightarrows \OO((X' {\times_X} X')^\sn)$. The same holds for any open $U \subseteq X$. That is, the canonical 
morphism $D \to X^{\sn}$ of locally ringed spaces is an isomorphism on topological spaces, \emph{and} structure sheaves. In other words, it is an isomorphism. 
Finally, note that we have a canonical inclusion of sheaves $\OO_X \subseteq \OO_C \subseteq \OO_{X^{\sn}}$. For any open affine $U \subseteq X$ of $X$, it 
follows that $\Spec(\OO_{X^{\sn}}(U)) \to \Spec(\OO_{C}(U)) \to \Spec(\OO_{X}(U))$ are homeomorphisms on topological spaces. Hence, $C$ is a scheme.
\end{proof}

\begin{prop}\label{prop:representable}
Suppose that the noetherian base scheme $S$ is Nagata. %
Then for every $X, Y \in \Schft[S]$ the canonical morphisms
\begin{equation} \label{equa:asgddbgas}
h^Y_\rh(X)=
h^Y(X^{\sn})
\end{equation}
are isomorphisms. 
The natural maps
\begin{equation} \label{equa:jksadf}
 h^Y_\rh\to h^Y_\cdh\to h^Y_\eh\to h^Y_\sdh\to h^Y_\val 
 \end{equation}
are isomorphisms of presheaves on $\Schft[S]$.
\end{prop}

\begin{proof}
We claim that $h^Y_\red$ is $\rh$-separated where
\[ h^Y_\red(-) = h^Y((-)_\red). \]
Let $f,g\in h^Y_\red(X)$ with
$f=g$ in $h^Y_\rh(X)$. In particular, $f_\eta=g_\eta$ at all generic points of $X$. But as $X_\red$ is reduced, this implies $f=g$. Since $h^Y_\red$ is a Zariski sheaf, in light of the factorisation of Remark~\ref{rem:defi:rhetc}\eqref{rem:defi:rhetc:refine}, we have 
\[ h^Y_\rh(X) = h^Y_\cdp(X) \]
when $X$ is reduced, cf. \cite[Prop.3.4.8(3)]{Kel12}, and hence, in general, as both $h^Y_\rh$ and $h^Y_\cdp$ are unchanged by reducing the structure sheaf.


By $\cdp$-separatedness of $h^Y_\red$ we have 
\begin{equation} \label{equa:lkhja}
h^Y_{\cdp}(X) = \varinjlim_{\cdp} \check{H}^0(X' / X, h^Y_\red) 
\end{equation}
where the colimit is over all $\cdp$-covers $p: X' {\to} X$. We claim that 
\begin{equation} \label{equa:lkhjagqwre}
\varinjlim_{\textrm{comp.dec.homeo.}} \check{H}^0(X' / X, h^Y_\red)  \to \varinjlim_{\cdp} \check{H}^0(X' / X, h^Y_\red) 
\end{equation}
is an isomorphism, where the first colimit is over completely decomposed homeomorphisms. Let $p: X' \to X$ be a $\cdp$-cover. For such a cover, define $X'' = \uSpec p_*\OO_{X'}$. Since $q: X' \to X''$ is proper, the topological space of $X''$ is the quotient of the topological space $X'$, via this morphism. Hence, any morphism $f' \in \eq(h^Y_\red(X') \to h^Y_\red(X' \times_Y X')$ factors through $X''$ as a a morphism of \emph{topological spaces}. But we have $q_* \OO_{X'} = \OO_{X''}$ by construction, and therefore $f$ comes from some $f'' \in h^Y_\red(X'')$. Then, since $(X'{\times_X}X')_\red \to (X''{\times_X}X'')_\red$ is dominant with reduced source, we actually have $f''\in \check{H}^0(X'' / X, h^Y_\red)$. Replacing $X''$ by $(X'')^\sn$ (this is where we use the assumption $S$ Nagata), we are in the situation of Lemma~\ref{lemm:cdpushout}, and find that $f''|_{(X'')^{\sn}}$ comes from some $g \in h^Y_{\red}(D)$ for some completely decomposed homeomorphism $D \to X$. As $(D{\times}_XD)_\red = D_\red$, we have $g \in \check{H}^0(D/X, h^Y_\red)$, so we have shown that \eqref{equa:lkhjagqwre} is surjective. It is clearly injective, as any refinement $X'' \to X' \to X$ of a completely decomposed homeomorphism by a $\cdp$-morphism is dominant. Hence, \eqref{equa:lkhjagqwre} is an isomorphism.

Finally, it follows from Lemma~\ref{lemm:seminormalProperties}\eqref{lemm:seminormalProperties:subInt} that $X^\sn \to X$ is an initial object in the category of completely decomposed homeomorphisms to $X$. So 
\[ h^Y_\cdp(X) = \varinjlim_{\textrm{comp.dec.homeo.}} \check{H}^0(X' / X, h^Y_\red) =\check{H}^0(X^{\sn} / X, h^Y_\red) = h^Y(X^{\sn}). \]

The isomorphisms \eqref{equa:jksadf}, except for $h_{\sdh}^Y$, are Theorem~\ref{thm:compare}. Injectivity of $h_{\sdh}^Y \to h_{\val}^Y$ has the same proof as injectivity of $h_{\eh}^Y \to h_{\val}^Y$, cf. the proof of Corollary~\ref{cor:injective}.

\end{proof}


\begin{rem}\ 
\begin{enumerate}
\item\label{rem:OhIsPerfectClosure}
This is analogous to the comparison,
$h^Y_{\h}(X)=h^Y(X^\sn)$ in characteristic zero, see \cite[Proposition~4.5]{HJ}, and \cite[Section~3.2]{Voe96}.

\item In fact, in general we have $h^Y_{\h}(X)=h^Y(X^\wn)$, where $Y^{\wn}$ is the absolute weak normalisation \cite[Def.B.1]{Ryd10}.
 For noetherian reduced schemes in pure positive characteristic, $Y^{\wn}$ is the perfect closure of $Y$ in $\prod_{i = 1}^n K_i^{a}$ the product of the
 algebraic closures of the function fields of its irreducible components. 
This holds much more generally: it is true for any algebraic space $Y$ locally
 of finite presentation, \cite[Thm.8.16]{Ryd10}.
\end{enumerate}
\end{rem}

In particular, the categories of representable $\rh$-, $\cdh$-, and $\eh$-sheaves on $\Schft[S]$ agree.

\begin{coro}[{cf. \cite[Thm.3.2.9]{Voe96}}]
Suppose the noetherian scheme $S$ is Nagata. 
The category of representable $\rh$-sheaves on $\Schft[S]$ is a localisation of the category $\Schft[S]$ with respect to completely decomposed homeomorphisms. In other words, it is obtained by formally inverting morphisms of the form $X_{\red} {\to} X$, then formally inverting subintegral extensions.
\end{coro}


\begin{proof}
Certainly, the functor $X \mapsto h^X_{\rh}$ factors through the localisation functor $(-)_{\red}: \Schft[S] \to (\Schft[S])_{\red}$. Now, it is straightforward to check that the class $\mathcal{S}$ of subintegral extensions of reduced schemes are a multiplicative system:
\begin{enumerate}
 \item $\mathcal{S}$ is closed under composition.
 \item For every $t :Z {\to} Y$ in $\mathcal{S}$ and $g:X{\to}Y$ in $(\Schft[S])_{\red}$ there is an $s: W {\to} X$ in $\mathcal{S}$ and $f: W{\to}Z$ in $(\Schft[S])_{\red}$ with $gs = tf$.
 \item If $f, g: X {\rightrightarrows} Y$ are parallel morphisms in $(\Schft[S])_{\red}$, then the following are equivalent:
 \begin{enumerate}
 \item $sf = sg$ for some $s \in \mathcal{S}$ with source $Y$.
 \item $ft = gt$ for some $t \in \mathcal{S}$ with target $X$.
\end{enumerate}
\end{enumerate}
(In fact, the latter two conditions are equivalent to $f {=} g$ in this case).

Since $\mathcal{S}$ is a multiplicative system, the hom sets in the localisation $\mathcal{S}^{-1}(\Schft[S])_{\red}$ are calculated by the formula 
\[ \hom_{\mathcal{S}^{-1}(\Schft[S])_{\red}}(X, Y) = \varinjlim_{X'{\to}X \in \mathcal{S}}\hom(X', Y)=\hom(X^\sn,Y),\]
 cf.~\cite{GZ67}, \cite[Thm.10.3.7]{Wei95}. But, by Proposition~\ref{prop:representable}, this is equal to $\hom_{\Shv_\rh(\Schft[S])}(h_{\rh}^X, h_{\rh}^Y)$. 
\end{proof}


\begin{coro}[{cf. \cite[Thm.3.2.10]{Voe96}}] \label{coro:hcdhleftAdjoint}
Suppose our notherian base scheme $S$ is Nagata. 
Let $\Shv_{\rh}^{\mathrm{rep}}(\Schft[S]) \subseteq \Shv_{\rh}(\Schft[S])$ denote the full subcategory of representable $\rh$-sheaves. The Yoneda functor $h^-_{\rh}: \Schft[S] \to \Shv_{\rh}^{\mathrm{rep}}(\Schft[S])$ admits a left adjoint. The counit of the adjunction is the seminormalisation $X^\sn \to X$. In particular, for any schemes $X, Y \in \Schft[S]$ with $X$ seminormal, one has 
\[ \hom(h_{\rh}^X, h_{\rh}^Y) = \hom_{\Schft[S]}(X, Y). \]
\end{coro}

\begin{proof}
We have $\hom_{\Shv_{\rh}}(h_{\rh}^X, h_{\rh}^Y) = h_{\rh}^Y(X) = h^Y(X^{\sn}) = \hom(X^\sn, Y)$.
\end{proof}

\section{Future directions}\label{sec:future}

\subsection{Relation to Berkovich spaces}

Berkovich \cite{berkovich} introduced a generalisation of rigid geometry in terms of seminorms. The sheaves $\Omega^n_{\valone}$ seem to be connected to Berkovich spaces. We review a very small part of the theory.

%
%

Recall that a \emph{multiplicative nonarchimedian norm} on a field $K$ is a group homomorphism $\av: K^* \to \RR_{>0}$ (the latter equipped with multiplication) such that $ |f + g| \leq \max(|f|, |g|)$. This is usually extended to a map of sets $\av: K \to \RR_{\geq 0}$ by setting $\av[0] = 0$.

\begin{key}
The set of multiplicative nonarchimedian norms on a field $K$ is the same as the set of pairs $(R, \av)$ where $R$ is a valuation ring of $K$, and $\av: K^* / R^* \to \RR_{>0}$ is an injective group homomorphism. Under this bijection, the ring $R$ corresponds to $\av^{-1}[0, 1]$. Since $\RR_{>0}$ has no convex subgroups, such valuation rings $R$ necessarily have rank 1 (or rank 0 if $R = K$).
\end{key}
\begin{proof}Obvious.\end{proof}

Let $K$ be a field equipped with a multiplicative nonarchimedian norm $||\cdot||: K^* \to \RR_{> 0}$. If $X$ is a $K$-variety, the \emph{Berkovich space} $X^\an$ of $X$, as a set, consists of pairs $(x, \av)$ where $x$ is a point of $X$, and $\av: \kappa(x) \to \RR_{>0}$ is a multiplicative nonarchimedian norm extending $||\cdot||$. This set is equipped with a structure of locally ringed space such that the projection $\pi: X^\an \to X; (x, \av) \mapsto x$ is a morphism of locally ringed spaces.

On the other hand, recall that Lemma~\ref{lemm:rsF} described $\Omega_\val^n(X)$ as
\begin{equation} \label{equa:rsFOmega}
\left \{ (s_x) \in \prod_{x \in X} \Omega^n(x) \middle | \begin{array}{c}
\textrm{ for every } k\textrm{-val.ring } R \subseteq k(x) \textrm{ of rank } \leq 1 
\textrm{ we have } \\
s_x \in \Omega^n(R) 
 \textrm{ and } 
s_x|_{R/\m} = s_y|_{R/\m}
\textrm{ where } y = \image(\Spec(R/\m)) 
\end{array} \right \}. 
\end{equation}

In particular, every section $s \in \Omega^n_{\val}(X)$ gives a function $X^\an \to \amalg_{(x, \av)} \Omega^n(\Hs(x))$ such that the image of $(x, |\cdot|)$ lands in the corresponding component. Here, $\Hs(x) = \widehat{\kappa(x)}$ is the completion of the normed field $\kappa(x)$. Similarly, we could apply $\Omega^n$ to the structure sheaf of the locally ringed space $X^\an$, and obtain a ring morphism $\Omega^n_{X^\an/K}(X^\an) \to \amalg_{(x, \av)} \Omega^n(\Hs(x))$.

\begin{ques}
Is the image of one of 
$\Omega^n_{X^\an/K}(X^\an) {\to} \amalg_{(x, \av)} \Omega^n(\Hs(x))$ 
and 
$\Omega^n_{\val}(X) {\to} \amalg_{(x, \av)} \Omega^n(\Hs(x))$ 
contained in the image of the other?
Does $\Omega^n_\val(X) {=} \Omega^n_{\cdh}(X)$ have an intrinsic description in terms of analytic spaces?
\end{ques}

\subsection{$F$-singularities and reflexive differentials} \label{sec:FsingReflexive}

Recall that reflexive differentials are defined as the double dual $\Omega^{[n]}_X = (\Omega^{n}_X)^{**}$. One of the results of \cite{HJ} is that on a 
klt base space $X$, $\cdh$-differentials recover reflexive differential forms; $\Omega_{X}^{[n]}(X) = \Omega^n_{\cdh}(X)$. 

On the other hand, there is an 
active area of research in positive characteristic birational geometry studying singularities defined via the Frobenius which are analogues of singularities 
arising in the minimal model program. These former are called $F$-singularities; log terminal, log canonical, rational, du Bois, correspond to $F$-regular, 
$F$-pure/$F$-split, $F$-rational, $F$-injective, respectively, cf. \cite[Remark 17.11]{Sch10}. Under this dictionary, Kawamata log terminal corresponds to 
strongly $F$-regular, \cite[Corollary 17.10]{Sch10}.

Consequently, we arrive at the following question.

\begin{ques}[Blickle]
If a normal scheme $X$ is strongly $F$-regular, do we have $\Omega_{X}^{[n]}(X) \cong \Omega^n_{\val}(X)$?
\end{ques}

In the special case $n = d_X = \dim X$, we have 
\[ \Omega^{d_X}_\val(X) 
= \varprojlim_{R \in \val(X)} \Omega^{d_X}(R) 
= \varprojlim_{R \in \val(X), R \subseteq k(X)} \Omega^{d_X}(R) 
= \Omega^{d_X}_{RZ(X)}(RZ(X)), \] 
and so the question becomes:

\begin{ques} \label{ques:blickTopDegree}
If a normal scheme $X$ of dimension $d_X$ is strongly $F$-regular, do we have $\Omega_{X}^{[d_X]}(X) \cong \varinjlim_{X' \to X} \Omega^{d_X}_{X'}(X')$? Here the colimit is over proper birational morphisms $X' \to X$.
\end{ques}

\begin{rema}
Under the assumption of resolution of singularities, Question~\ref{ques:blickTopDegree} is true: being strongly $F$-regular implies being pseudo-rational, which means (by definition) that for any proper birational morphism $\pi: X' \to X$ the direct image $\pi_{*}\omega_{X'}$ of the canonical dualizing sheaf $\omega_{X'}$ of $X'$ is $\omega_X$, the canonical dualizing sheaf of $X$. If $X'$ is smooth over the base, we have $\omega_{X'} = \Omega^{d_X}_{X/k}$. On the other hand, we also have $\Omega^{[d_X]}_X = \omega_X$. So if we restrict the colimit to those $X'$ which are smooth, we have
\[ \Omega^{[d_X]}_X 
= \omega_X 
= \varinjlim_{\substack{X' \to X \\ X' \mathrm{ smooth}}} \omega_{X'}(X') 
= \varinjlim_{\substack{X' \to X \\ X' \mathrm{ smooth}}} \Omega^{d_X}_{X'}(X'). 
 \]
Under the assumption of resolution of singularities, colimit over $X'$ smooth is the same as the colimit over all $X'$. 
\end{rema}

We remark that an alternative description of reflexive differentials when $X$ is klt of characteristic zero is given in \cite{GKP14} by $\Omega^{[p]}_X \cong \pi_*\Omega^p_{X'}$ where $\pi: X' \to X$ is a log resolution, and that this implies an isomorphism $\Omega^{[d_X]}_X(X) \cong \Omega^{d_X}_{\val}(X)$ (in characteristic zero) where $d_X = \dim X$.

Let us list some facts about strongly $F$-regular schemes that seem relevant to the current discussion. To begin with, a ring $R$ is \emph{strongly $F$-regular} if and only if its test ideal is equal to $R$, \cite[Proposition 16.9]{Sch10}. This implies that $R$ is Cohen-Macauley, and in particular, that the sheaf $\omega_R$ is a dualising object in the derived category.

About test ideals: It is shown in \cite{BST} that in positive characteristic, the test ideal of a normal variety $X$ over a perfect field can be defined as the intersection of the images of certain trace maps, with the intersection taken over all generically finite proper separable maps $\pi: Y \to X$ with $Y$ regular. Here, the trace map comes from the trace morphism $\pi_!\pi^! \omega_X^\bullet \to \omega_X^\bullet$. Moreover, there exists a $Y \to X$ in the indexing set, whose image agrees with this intersection.

So in the affine case the above question becomes:

\begin{ques}
Let $X$ be a normal affine variety over a perfect field, with $\Q$-Cartier canonical divisor. 
Let $\iota: X_\reg \to X$ be the inclusion of the regular locus.
Suppose that there exists a generically finite proper separable morphism $\pi: Y \to X$ with regular source such that 
\[ \OO_X \cong \image \biggl ( \pi_*\OO_Y(K_Y {-} \pi^*K_X) \stackrel{\textrm{trace}}{\longrightarrow} K(X) \biggr ). \]
Is the canonical morphism
\[ \Omega_{\val}^n(X) \to \Omega_{\val}^n(X_\reg) \]
an isomorphism?
\end{ques}

In this formulation, we have used that for any normal scheme $X$ with regular locus $\iota: X_\reg \to X$ one has $\Omega^{[n]}_X = \iota_*\Omega_{X_\reg}^n$. %
Since we know that $\Omega^n = \Omega^n_{\val}$ on regular schemes, we obtain the description $\Omega^{[n]}_X(X) = (\iota_*\Omega_{X_\reg}^n)(X) = \Omega_{X_\reg}^n({X_\reg}) = \Omega^n_{\val}({X_\reg})$.


\bibliographystyle{alpha}
\bibliography{val}

\end{document}